\documentclass[11pt,a4paper]{article}
\usepackage[utf8]{inputenc}
\usepackage[english]{babel}
\usepackage{arydshln}

\usepackage{amsthm}

\newcommand{\G}{\Gamma}
\newcommand{\vG}{\vec{\Gamma}}
\newcommand{\Aut}{\mathrm{Aut}}
\newcommand{\ZZ}{\mathbb{Z}}
\newcommand{\PSL}{\mathrm{PSL}}

\newtheorem{theorem}{Theorem}[section]
\newtheorem{proposition}[theorem]{Proposition}
\newtheorem{lema}[theorem]{Lemma}
\newtheorem{corollary}[theorem]{Corollary}

\newtheorem{question}[theorem]{Question}

\theoremstyle{definition}
\newtheorem{agreement}[theorem]{Agreement}

\usepackage{xcolor}
\usepackage{amssymb}
\usepackage{amsmath}
\usepackage{graphicx}
\usepackage{enumerate}
\usepackage{verbatim}
\usepackage{enumitem}
\usepackage{subfig}

\if@twoside \oddsidemargin 5pt \evensidemargin 5pt \marginparwidth 20pt
 \else \oddsidemargin 5pt \evensidemargin 5pt \marginparwidth 20pt \fi

\begin{document}

%\overset{\rightharpoonup}{\Gamma} ali \vec{\Gamma} - eno harpuna in rabi vec prostora, drugo puscica in ne dela vecjiega razmaka

\begin{center}
{\bf\large
On tetravalent half-arc-transitive graphs of girth $5$}
\end{center}

%\bigskip

\begin{center}
Iva Anton\v ci\v c{\small $^{a,*}$},  
Primo\v z \v Sparl{\small $^{b,c,d,}$}\footnotemark
\\

\medskip
{\it {\small
$^a$University of Primorska, The Faculty of Mathematics, Natural Sciences and Information Technologies, Koper, Slovenia\\
$^b$University of Ljubljana, Faculty of Education, Ljubljana, Slovenia\\
$^c$University of Primorska, Institute Andrej Maru\v si\v c, Koper, Slovenia\\
$^d$Institute of Mathematics, Physics and Mechanics, Ljubljana, Slovenia
}}
\end{center}

%\addtocounter{footnote}{1}
\footnotetext{
This work is supported in part by the Slovenian Research Agency
(research program P1-0285
and research projects {J1-1694, J1-1695, J1-2451 and J1-3001).}

~*Corresponding author e-mail:~antoncic.iva@gmail.com}

%\vspace*{-19pt}

\begin{abstract}
A subgroup of the automorphism group of a graph $\G$ is said to be {\em half-arc-transitive} on $\G$ if its action on $\G$ is transitive on the vertex set of $\G$ and on the edge set of $\G$ but not on the arc set of $\G$. Tetravalent graphs of girths $3$ and $4$ admitting a half-arc-transitive group of automorphisms have previously been characterized. In this paper we study the examples of girth $5$. We show that, with two exceptions, all such graphs only have directed $5$-cycles with respect to the corresponding induced orientation of the edges. Moreover, we analyze the examples with directed $5$-cycles, study some of their graph theoretic properties and prove that the $5$-cycles of such graphs are always consistent cycles for the given half-arc-transitive group. We also provide infinite families of examples, classify the tetravalent graphs of girth $5$ admitting a half-arc-transitive group of automorphisms relative to which they are tightly-attached and classify the tetravalent half-arc-transitive weak metacirculants of girth $5$.
\end{abstract}

%\vspace*{-19pt}

\begin{quotation}
Math. Subj. Class.: 05C25, 05C38, 20B25.
\end{quotation}

\begin{section}{Introduction}

Throughout this paper graphs are assumed to be finite, simple, connected and undirected (but with an implicit orientation of the edges when appropriate). Most of the relevant group and graph theoretic concepts are reviewed in Section~\ref{sec:prelim}, but for the ones that are not see~\cite{DF04book, GR01book}. 

The topic of half-arc-transitive graphs and graphs admitting half-arc-transitive group actions has been extensively researched in the past decades with numerous publications and with researchers focusing on several different aspects of such graphs (see for instance~\cite{CuiZho21, PozPra21, RamSpa19} and the references therein). The vast majority of papers dealing with half-arc-transitive graphs focuses on examples with smallest possible valence, namely on the tetravalent ones, which is also what we do in this paper. One of the more fruitful approaches, in the context of tetravalent graphs in particular, has been studying the structure imposed on the graph by the half-arc-transitive action of the corresponding group. Namely, it is easy to see (but see for instance~\cite{MTAO}) that a half-arc-transitive subgroup of the automorphism group of a graph induces two paired orientations of its edges. One can then study the structural properties of the given tetravalent graph induced by the so called alternating cycles (the investigation of which was initiated in~\cite{MTAO} and recently taken further in~\cite{RamSpa19}), or by the so called reachability relations (introduced in a more general context in~\cite{MarPot02} and then further studied in~\cite{MalMarSeiSpaZgr08, MalPotSeiSpa15} also for infinite (di)graphs). 

Investigations with respect to alternating cycles have been particularly successful, resulting in partial classifications and characterizations of such graphs (see for instance~\cite{MTAO, MarPra99, PotSpa17, TA4E, Wil04}), and revealing that the tetravalent half-arc-transitive weak metacirculants (first studied in~\cite{classes} and later analyzed in a number of papers, see for instance~\cite{ClassIV, CuiZho21, OCM, C2C4}) form an important subfamily of these graphs. 

Another situation, where the local structure given by a half-arc-transitive group action proves to be very useful, is when one studies examples with small girth. It was observed back in 1997~\cite{MarXu97} that the tetravalent half-arc-transitive graphs of girth $3$ naturally correspond to the $1$-regular cubic graphs (cubic graphs whose automorphism group acts regularly on the set of their arcs). Obtaining a nice characterization of half-arc-transitive tetravalent graphs of girth $4$ turned out to be more challenging~\cite{HTG4}. The weak metacirculants among them were later classified in~\cite{ClassIV}. It should be mentioned that tetravalent edge-transitive graphs of girths $3$ and $4$ were thoroughly investigated also in~\cite{PotWil07}.
\medskip

In this paper we focus on the examples of girth $5$. More precisely, we study tetravalent graphs of girth $5$ admitting a half-arc-transitive group action. We study the possible types of $5$-cycles with respect to the induced orientation of the edges that can arise (see Section~\ref{sec:types} for the definition) and provide a complete classification of the examples with undirected $5$-cycles (see Theorem~\ref{neusmerjeni}). Quite remarkably, it turns out that there are just two such graphs and that one of them is in fact the smallest half-arc-transitive graph (the well-known Doyle-Holt graph) while the other is a certain Rose window graph admitting two different half-arc-transitive subgroups of automorphisms giving rise to $5$-cycles of two different types. 

All other tetravalent graphs of girth $5$ admitting half-arc-transitive group actions have directed $5$-cycles in the corresponding orientations of the edges and here the situation is much more complex and diverse. For these graphs we first provide a classification of the tightly-attached examples (see Theorem~\ref{TA}), where it is interesting to note that all such examples are in fact half-arc-transitive (that is, their full automorphism group is half-arc-transitive). The non-tightly-attached examples are then studied in Section~\ref{sec:directed}. We prove various properties of such graphs, among them the fact that in this case the $5$-cycles are all consistent and that each edge lies on at most two $5$-cycles, to name just two. We provide infinitely many examples having precisely one $5$-cycle through any given edge and also infinitely many examples having two $5$-cycles through any given edge (see Theorem~\ref{the:oneortwo}). We also show that for each of the two possibilities half-arc-transitive examples exist. Interestingly, the smallest half-arc-transitive one with two $5$-cycles through each edge is quite large with its order being $6480$. We also analyze the examples that can be found in the census of all tetravalent graphs up to order $1000$ admitting a half-arc-transitive action~\cite{censusHAT} (see also~\cite{PotWil}) to compare some other properties of such graphs. It turns out that there is a great variety in their properties (see the tables in Section~\ref{sec:directed}) which suggests that a complete classification of all tetravalent graphs of girth $5$ admitting a half-arc-transitive action is probably not feasible. We conclude the paper by an application of our results yielding a complete classification of tetravalent half-arc-transitive weak metacirculants of girth $5$ (see Theorem~\ref{the:meta}).

\end{section}

\begin{section}{Preliminaries}
\label{sec:prelim}

For a graph $\G$ we let $V(\Gamma)$, $E(\Gamma)$ and $A(\Gamma)$ be its vertex set, edge set and arc set, respectively, where an {\em arc} is an ordered pair of adjacent vertices. A subgroup $G \leq \Aut(\G)$ is said to be {\em vertex-transitive}, {\em edge-transitive} and {\em arc-transitive} on $\G$ if the corresponding natural action of $G$ on $V(\G)$, $E(\G)$ and $A(\G)$, respectively, is transitive. If $G$ is vertex- and edge- but not arc-transitive on $\G$, it is said to be {\em half-arc-transitive} on $\G$. In this case $\G$ is said to be {\em $G$-half-arc-transitive}. In the case that $G = \Aut(\G)$ we simply say that $\G$ is {\em half-arc-transitive}. Observe that if $\G$ admits a half-arc-transitive subgroup of automorphisms then $\G$ is either half-arc-transitive or is arc-transitive. We will often be using the well-known fact (see~\cite[Proposition 2.1]{MTAO}) that if $G \leq \Aut(\G)$ is half-arc-transitive on $\G$ then no element of $G$ can interchange a pair of adjacent vertices of $\G$.

As is well known (but see for instance~\cite{MTAO}) a half-arc-transitive action of $G \leq \Aut(\G)$ induces two paired orientations of edges of $\G$. If $\vec{\G}$ is one of the two corresponding oriented graphs with underlying graph $\G$ and $(u,v)$ is an arc of $\vec{\G}$, we say that $u$ is the {\em tail} and $v$ is the {\em head} of the arc $(u,v)$ and indicate this by writing $u \to v$. In the case that $\G$ is tetravalent (which will be the case throughout this paper) each vertex is the tail of two and the head of two arcs of $\vec{\G}$, thereby giving rise to {\em $G$-alternating cycles} which are cycles of $\G$ for which every two consecutive edges have opposite orientations in $\vec{\G}$. It was proved in~\cite{MTAO} that all $G$-alternating cycles of a graph have the same even length, half of which is called the {\em $G$-radius} of $\G$, and that every pair of nondisjoint $G$-alternating cycles meets in the same number of vertices (called the {\em $G$-attachment number} of $\G$). If the $G$-attachment number coincides with the $G$-radius, the graph is said to be {\em $G$-tightly-attached} (in all of the above terminology we omit the prefix $\Aut(\G)$ when $G = \Aut(\G)$). 

The tightly-attached examples were classified in~\cite{MTAO, TA4E} (but see also~\cite{MarPra99, Wil04}) and in~\cite{classes} it was shown that they represent Class~I of the four (not disjoint) classes of a larger family of {\em weak metacirculants}, which are graphs admitting a transitive metacyclic group. In the case that the quotient graph with respect to the orbits of the corresponding normal cyclic subgroup of the metacyclic group is of valence $4$ the graph is said to be of Class~IV. The results of~\cite{C2C4} imply that each tetravalent half-arc-transitive weak metacirculant belongs to at least one of Classes~I and~IV and the tetravalent half-arc-transitive weak metacirculants of Class~IV and girth $4$ were classified in~\cite{ClassIV}.
\medskip

When studying the structure of (tetravalent) graphs admitting a half-arc-transitive group of automorphisms the concept of reachability can also be useful. Here we only describe some of the corresponding notions first introduced in~\cite{MarPot02} and later further studied in~\cite{MalMarSeiSpaZgr08, MalPotSeiSpa15} but see~\cite{MarPot02, MalMarSeiSpaZgr08} for details. Let $\vG$ be an oriented graph (in our case we will always take one of the two $G$-induced orientations of the tetravalent graph in question where $G$ will be a corresponding half-arc-transitive group of automorphisms) and let $W = (v_0, v_1, \ldots , v_m)$ be a walk in the underlying graph of $\vG$. The {\em weight} $\zeta(W)$ of $W$ (called the sum of $W$ in~\cite{MarPot02}) is defined as the number of $i$, $0 \leq i < m$, such that $(v_i,v_{i+1}) \in A(\vG)$, minus the number of $i$, $0 \leq i < m$, such that $(v_{i+1},v_i) \in A(\vG)$. We then say that the vertices $u$ and $v$ of $\vG$ are {\em $R^+$-related} if there exists a walk $W = (v_0,v_1,\ldots , v_n)$ in $\vG$ from $u$ to $v$ such that $\zeta(W) = 0$ and for each $i$ with $0 \leq i \leq n$ the weight of the subwalk $(v_0,v_1, \ldots , v_i)$ of $W$ is nonnegative. The number of equivalence classes of $R^+$ is called the {\em alter-perimeter} of $\vG$ and $\vG$ is said to be {\em alter-complete} if its alter-perimeter is $1$ (that is, if all pairs of vertices are $\mathrm{R}^+$-related). 
\medskip

When studying graphs possessing a large degree of symmetry the notion of consistent cycles turns out to be useful. Here we only review the notions and results needed in this paper, but see~\cite{CC} where consistent cycles in half-arc-transitive graphs were thoroughly investigated. Let $G \leq \Aut(\G)$ be a vertex-transitive subgroup of automorphisms of a graph $\G$. A cycle $C$ of $\G$ is said to be {\em $G$-consistent} if there exists an element $g \in G$ (which is said to be a {\em shunt} of $C$) acting as a $1$-step rotation of $C$. It is easy to see (but see~\cite{CC}) that $G$ has a natural action on the set of all $G$-consistent cycles of $\G$ and so one can speak of $G$-orbits of $G$-consistent cycles. It follows from~\cite[Proposition 5.3]{CC} that if $G$ is a half-arc-transitive subgroup of automorphisms of a tetravalent graph then there are precisely two $G$-orbits of $G$-consistent cycles of $\G$.
\medskip

We conclude this section by reviewing the definition of an important family of tetravalent graphs, known as the Rose window graphs~\cite{Wil08}, containing four infinite subfamilies of vertex- and edge-transitive (in fact arc-transitive) examples (see~\cite{KovKutMar10}). For an integer $n \geq 3$ and integers $a, r$, where $1 \leq r < n/2$ and $1 \leq a \leq n-1$, the {\em Rose window} graph $\mathrm{R}_n(a, r)$ has vertex set $V = \{x_i \, | i\in \mathbb{Z}_{n}\}\cup \{y_i\, | i\in \mathbb{Z}_{n}\}$ and edge set $E=\{ \{x_i,\, x_{i+1}\} | i\in \mathbb{Z}_{n}\}\cup \{ \{y_i,\, y_{i+r}\} | i\in \mathbb{Z}_{n}\}\cup \{ \{x_i,\, y_{i} \} | i\in \mathbb{Z}_{n}\}\cup \{ \{x_i,\, y_{i-a}\} | i\in \mathbb{Z}_{n}\}$, where $\ZZ_n$ is the residue class ring modulo $n$ and where indices of vertices are computed modulo $n$.

\end{section}

%%%%%%%%%%%%%%%%%%%%%%%%%%%%%%%%%%%%%%%%%%%%%%%%%%%%%%%%%%%%%%%%%

\begin{section}{Types of $5$-cycles}
\label{sec:types}

Throughout this section let $\Gamma$ be a tetravalent graph of girth $5$ admitting a half-arc-transitive subgroup $G\leq\text{Aut}(\Gamma )$ and let $\vec{\G}$ be one of the two $G$-induced orientations of the edges of $\Gamma$. 
Let $C$ be a $5$-cycle of $\Gamma$. We define its {\em type} with respect to $\vG$ depending on the corresponding orientations of its edges in $\vec{\Gamma}$ as follows. Suppose first that at least one vertex $v$ of $C$ is either the head or the tail of the two arcs of $\vec{\Gamma}$ corresponding to the two edges of $C$ incident to $v$. In this case, letting $s$ be the largest integer such that $C$ contains a directed path of length $s$ of $\vec{\Gamma}$ (note that $2 \leq s \leq 4$), we say that $C$ is of {\em type $s$} with respect to $\vec{\Gamma}$ and that it is an {\em undirected} $5$-cycle with respect to $\vec{\Gamma}$. The only other possibility is that each vertex $v$ of $C$ is the head of one and the tail of the other of the two arcs of $\vG$ corresponding to the two edges of $C$ incident to $v$. Such $5$-cycles are said to be of {\em type $5$} and are said to be {\em directed} with respect to $\vG$. The four potential types of $5$-cycles are presented in Figure~\ref{5cycles}. 
\begin{figure}[ht]
\begin{center}
\includegraphics[height=3 cm]{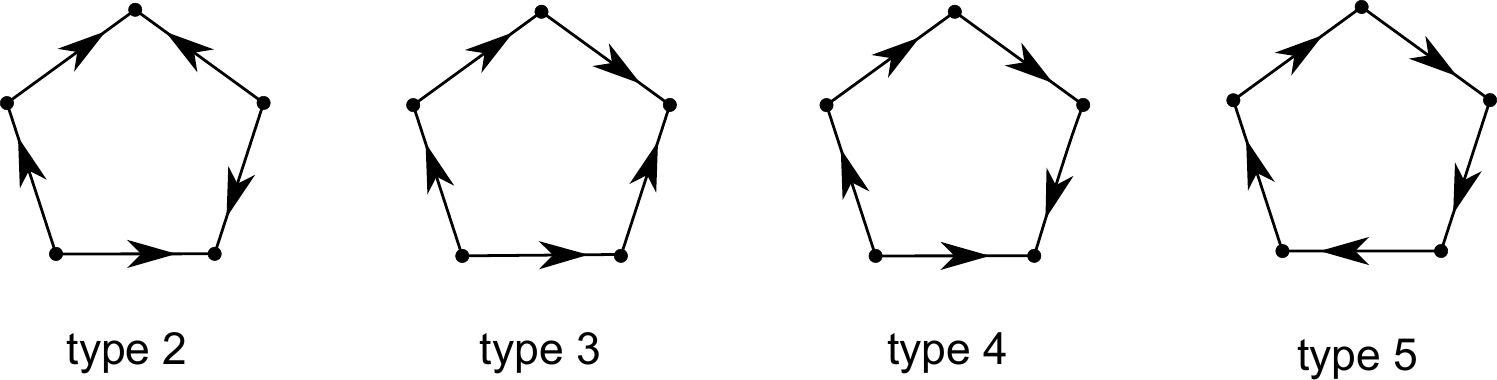}
\caption{Possible types of $5$-cycles.}
\label{5cycles}
\end{center} 
\end{figure}

Observe that taking the other of the two $G$-induced orientations of the edges of $\G$ does not change the type of a $5$-cycle. This is why henceforth, whenever a half-arc-transitive group $G \leq \Aut(\G)$ is clear from the context, we simply speak of directed and undirected $5$-cycles of $\G$ and of $5$-cycles of $\G$ of types $2$, $3$, $4$ and $5$. Observe also that for each $s$ with $2 \leq s \leq 4$ each $5$-cycle of type $s$ contains precisely one directed path of length $s$ of $\vG$. This $s$-path will be called {\em the $s$-arc} of this $5$-cycle. 
\medskip

For ease of reference we first record some rather obvious facts about $5$-cycles in tetravalent graphs admitting a half-arc-transitive group of automorphisms. 

\begin{lema}\label{basicProperties}
Let $\Gamma$ be a tetravalent graph of girth $5$ admitting a half-arc-transitive subgroup $G\leq\text{Aut}(\Gamma )$ and let $\vec{\Gamma}$ be one of the two $G$-induced orientations of $\Gamma$. Then the following all hold. 
\begin{enumerate}[label=(\roman*)]
\item \label{i} The elements of $G$ preserve types of $5$-cycles of $\Gamma$.
\item \label{ii} A $3$-arc of $\Gamma$ lies on at most one $5$-cycle of $\Gamma$.
\item \label{iii} The closed neighborhood of any $5$-cycle consists of $15$ vertices.
\end{enumerate}
\end{lema}

\begin{proof}
The first claim follows from the assumption that the orientation $\vec{\Gamma}$ is $G$-induced, while the other two follow from the assumption that $\G$ has no cycles of length less than $5$. 
\end{proof}

Lemma~\ref{basicProperties}\ref{iii} implies that if $C = (v_0,v_1,v_2,v_3,v_4)$ is a $5$-cycle of $\G$ then the closed neighborhood $N[C]$ of $C$ consists of $15$ vertices. For each $i \in \ZZ_5$ we can thus denote the two neighbors of $v_i$ outside $C$ by $u_i$ and $w_i$ (see the left-hand side of Figure~\ref{fig:okolica}). In fact, it is easy to see that we can assume these labels have been chosen in such a way that for each $i \in \ZZ_5$ either $(v_{i-1},v_i,u_i)$ or its reverse is a directed $2$-path of $\vG$ and at the same time either $(w_i,v_i,v_{i+1})$ or its reverse is a directed $2$-path of $\vG$, where all the indices are computed modulo $5$ (see the right-hand side of Figure~\ref{fig:okolica} for an example with a $5$-cycle of type~3). We record this agreement for ease of future reference.
\begin{figure}[ht]
\begin{center}
\includegraphics[height=4 cm]{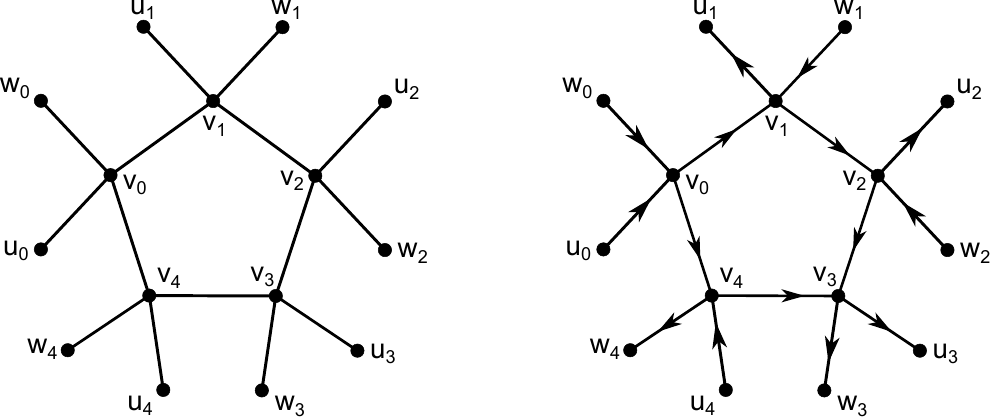}
\caption{The closed neighborhood of a $5$-cycle and orientations of the edges $u_iv_i$ and $v_iw_i$.}
\label{fig:okolica}
\end{center} 
\end{figure}

\begin{agreement}
\label{okolicaC}
If $\G$ is a tetravalent graph of girth $5$ admitting a half-arc-transitive subgroup $G \leq \Aut(\G)$, $\vG$ is one of the two $G$-induced orientations of the edges of $\G$ and $C = (v_0,v_1,v_2,v_3,v_4)$ is a $5$-cycle of $\G$, we denote the vertices of $N[C]$ in such a way that 
$$
	N[C] = \{v_i \colon i \in \ZZ_5\} \cup \{u_i \colon i \in \ZZ_5\} \cup \{w_i \colon i \in \ZZ_5\}, 
$$
and that for each $i \in \ZZ_5$ one of $(v_{i-1},v_i,u_i)$ or its reverse and one of $(w_i,v_i,v_{i+1})$ or its reverse is a directed $2$-path of $\vG$.
\end{agreement}

We conclude this section by another observation that will be useful in the subsequent sections.

\begin{lema}\label{stabilizerSize}
Let $\Gamma$ be a tetravalent graph of girth $5$ admitting a half-arc-transitive subgroup $G\leq\text{Aut}(\Gamma )$ and let $\vec{\Gamma}$ be one of the two $G$-induced orientations of $\Gamma$. Then the natural action of $G$ on the set of all $2$-arcs of $\vec{\Gamma}$ is transitive if and only if $|G_v| > 2$ for each $v\in V(\Gamma )$. Moreover, if $|G_v|=2$ for each $v\in V(\Gamma )$ then $G$ has precisely two orbits on the set of all $2$-arcs of $\vec{\Gamma}$, for each $v\in V(\Gamma )$ there are precisely two $2$-arcs with middle vertex $v$ from each of these two $G$-orbits, and the two $2$-arcs with middle vertex $v$ from the same $G$-orbit have no edges in common.
\end{lema}

\begin{proof}
This is a straightforward consequence of the fact that in $\vG$ each vertex $v$ has two out-neighbors and two in-neighbors and that $G$ acts transitively on the set of arcs of $\vG$. Namely, this implies that for each $v \in V(\G)$ there is an element of $G_v$ interchanging the two out-neighbors of $v$ and at the same time interchanging the two in-neighbors of $v$. Therefore, $G$ acts transitively on the set of $2$-arcs of $\vG$ if and only if there is an element of $G_v$ fixing the two out-neighbors of $v$ while interchanging the two in-neighbors of $v$. We will see that such an automorphism of $G_v$ exists if and only if $|G_v| > 2$. Obviously in the case of its existence, $|G_v|>2$ holds. Let now $|G_v| > 2$. Then there exists a nontrivial automorphism $\alpha \in G_v$ fixing at least one of the neighbors of $v$. Since the two in-neighbors are in the same $G_v$-orbit  and the two out-neighbors are in the same $G_v$-orbit, this means that $\alpha$ either fixes both in-neighbors, or both out-neighbors. If it also interchanges the out-neighbors in the first case or in-neighbors in the second case, the sufficiency is proven. If not, then since $\G$ is connected and $\alpha\ne id$, there is some vertex fixed by $\alpha$, were $\alpha$ does not fix all its neighbors, but does fix both in-neighbors or both out-neighbors. Since $G$ acts transitively on the vertex set of $\G$, all vertex stabilizers are isomorphic, meaning that this is also true for $G_v$, as required.

\end{proof}

\end{section}

%%%%%%%%%%%%%%%%%%%%%%%%%%%%%%%%%%%%%%%%%%%%%%%%%%%%%%%%%%%%%%%%%%%%%%%%%%%%%%%%%%%%%%%%%%%%%%%%%%%%%%%%%%%%%%%%%%%%%%%%%%%%%%%%%%%%%%%%%%%%%%%%%%%%%%%%%%%%%

\begin{section}{Tightly-attached $G$-half-arc-transitive graphs of girth $5$}
\label{sec:TA}

Before analyzing the possible four types of $5$-cycles in all generality we determine the tetravalent graphs of girth $5$ admitting a half-arc-transitive group action relative to which they are tightly-attached.

Let $\Gamma$ be a connected tetravalent graph of girth $5$ admitting a half-arc-transitive subgroup $G\leq\text{Aut}(\Gamma )$ for which it is tightly-attached. By \cite[Proposition 2.4.]{MTAO} it has at least three $G$-alternating cycles (otherwise it is a Cayley graph of a cyclic group and as such has girth at most $4$). Furthermore, as $\Gamma$ is of girth $5$ its $G$-radius $r_G$ is at least $3$. We first show that $r_G$ must be odd. If it is even, then by~\cite[Theorem 1.2.]{TA4E} the graph $\G$ is isomorphic to some $\mathcal{X}_e(m,\, r;\, q,\, t)$ where $m$ is even (see \cite{TA4E} for the definition of $\mathcal{X}_e$), implying that it is bipartite. This contradicts the assumption that $\G$ is of girth $5$. It thus follows that $r_G$ is odd, as claimed. 

To be able to state our result we recall the definition of the $\mathcal{X}_o (m,\, r;\, q)$ graphs from \cite{MTAO, TA4E} (we mention that $r$ is the $G$-radius of the graph). For integers $m\geq 3$ and $r\geq 3$, where $r$ is odd, and for $q\in \mathbb{Z}_r$ with $q^m=\pm 1\>$ $(\text{in}\>  \mathbb{Z}_r)$ let $\mathcal{X}_o(m,\, r;\, q)$ be the graph with vertex set  $V=\{ u_i^j;\, i\in \mathbb{Z}_m,\, j\in\mathbb{Z}_r\}$ and the edges defined by the adjacencies 
$$
	u_i^j\sim u_{i+1}^{j\pm q^i}\, ;\ i\in\mathbb{Z}_m ,\ j\in \mathbb{Z}_r,
$$
where the subscripts are computed modulo $m$ and the superscripts modulo $r$.
\medskip

By \cite[Proposition 3.3.]{MTAO} each connected tetravalent graph $\Gamma$ admitting a half-arc-transitive subgroup $G\leq\text{Aut}(\Gamma )$ such that $\Gamma$ is $G$-tightly-attached with an odd $G$-radius $r$ is isomorphic to $\mathcal{X}_o(m,\, r;\, q)$ for some $m\geq 3$ and $q\in \mathbb{Z}_r$ such that $q^m=\pm 1$. In fact, since $m$ is odd and \cite[Proposition 4.1.]{MTAO} ensures that $\mathcal{X}_o (m,\, r;\, q)\cong \mathcal{X}_o (m,\, r;\, -q)$, we can actually assume that $q^m=1$. Suppose now that $\Gamma$ is of girth $5$. Then $q\neq \pm 1$ (otherwise $\Gamma$ has $4$-cycles). By \cite[Proposition 3.1.]{MTAO} one of the two $G$-induced orientations of the edges of $\G$ is such that for each $i\in\mathbb{Z}_m$ and each $j\in\mathbb{Z}_r$ the edges $u^j_iu^{j\pm q^i}_{i+1}$ are both oriented from $u^j_i$ to $u^{j+ q^i}_{i+1}$ and $u^{j- q^i}_{i+1}$. Therefore, only $5$-cycles of types $4$ and $5$ are possible. Moreover, in the case of type $4$ we must have $m=3$ with $3 \pm q \pm q^2 = 0 $, while in the case of $5$-cycles of type 5 we have that $m=5$ and $1 \pm q \pm q^2 \pm q^3 \pm q^4 = 0$. 

First we analyze the possibility of having $5$-cycles of type $4$. In this case $\Gamma\cong\mathcal{X}_o(3,\, r;\, q)$ with $q^3= 1$ and $3\pm q\pm q^2=0$. Using division-free row reductions of the corresponding Sylvester matrices (see \cite[Section 4]{TA4E}) we see that the only possibility (so as to avoid $q=\pm 1$) is that $r = 9$ and $q \in \{4,7\}$ resulting in the well-known Doyle-Holt graph $\mathcal{X}_o (3,\, 9;\, 4)\cong \mathcal{X}_o (3,\, 9;\, 7)$. 

In the case of directed $5$-cycles we have $\Gamma\cong\mathcal{X}_o(5,\, r;\, q)$ with $q^5=1$ and $1 \pm q \pm q^2 \pm q^3 \pm q^4 = 0$. As above, division-free row reductions of the corresponding Sylvester matrices show that the only possibility to avoid $q=\pm 1$ is that $1+q+q^2+q^3+q^4=0$. Note that $q\neq \pm 1$ together with $q^5 = 1$ forces $q^2\neq \pm 1$. Then \cite[Theorem 3.4.]{MTAO} implies that $\Gamma$ is in fact half-arc-transitive. We have thus established the following result.

\begin{theorem}
\label{TA}
A tetravalent graph $\Gamma$ of girth $5$ is a tightly-attached $G$-half-arc-transitive graph for some subgroup $G\leq\text{Aut}(\Gamma )$ if and only if one of the following holds:
\begin{itemize}
\item $\Gamma\cong \mathcal{X}_o(3,\, 9;\, 4)$, the well-known Doyle-Holt graph;
\item $\Gamma\cong \mathcal{X}_o(5,\, r;\, q)$, where $q\ne \pm 1$ and $1+q+q^2+q^3+q^4=0$.
\end{itemize}
Moreover, all such graphs are half-arc-transitive. Furthermore, the only tetravalent graph of girth $5$ admitting a half-arc-transitive group of automorphisms relative to which it is tightly-attached and has undirected $5$-cycles is the Doyle-Holt graph.
\end{theorem}

Given that there is only one tetravalent graph of girth $5$ admitting a half-arc-transitive group action relative to which it is tightly-attached and the $5$-cycles are undirected, the natural question is whether there are other such graphs if we omit the condition of the graph being tightly-attached. We answer this question in the next section.

 \end{section}

%%%%%%%%%%%%%%%%%%%%%%%%%%%%%%%%%%%%%%%%%%%%%%%%%%%%%%%%%%%%%%%%%%%

\begin{section}{$G$-half-arc-transitive graphs with undirected $5$-cycles}
\label{sec:undir}

Throughout this section let $\Gamma$ be a tetravalent $G$-half-arc-transitive graph of girth $5$ for some $G\leq\text{Aut}(\Gamma )$ such that in a chosen $G$-induced orientation $\vec{\Gamma}$ of the edges of $\Gamma$ there are undirected $5$-cycles, that is $5$-cycles of type $2$, $3$ or $4$.

We first show that $5$-cycles of type~2 in fact cannot occur. By way of contradiction suppose $C = (v_0,\, v_1,\, v_2,\, v_3,\, v_4)$ is a $5$-cycle of type $2$ of $\G$ where $(v_4,\, v_0 ,\, v_1)$ is the $2$-arc of $C$. Each element of $G$ that fixes the vertex $v_2$ and interchanges the vertices $v_1$ and $v_3$ interchanges the vertices $v_0$ and $v_4$, contradicting the fact that no element of $G$ can interchange a pair of adjacent vertices. Therefore the following holds.

\begin{proposition}\label{no2}
Let $\G$ be a tetravalent $G$-half-arc-transitive graph of girth $5$ for some $G \leq \Aut(\G)$. Then $\G$ has no $5$-cycles of type $2$.
\end{proposition}

Before analyzing the other two types of undirected $5$-cycles let us first prove a useful fact about the stabilizer of a vertex in the action of $G$.

\begin{proposition}\label{z2}
Let $\Gamma$ be a tetravalent graph of girth $5$ admitting a half-arc-transitive subgroup $G\leq\text{Aut}(\Gamma )$ such that a corresponding $G$-induced orientation $\vec{\Gamma}$ gives rise to undirected $5$-cycles. Then $G_v\cong \mathbb{Z}_2$ for all $v\in V(\Gamma)$.
\end{proposition}

\begin{proof}
By Proposition \ref{no2} we only have to consider the cases where $\vec{\Gamma}$ has $5$-cycles of type $3$ or of type $4$. Let $C = (v_0,v_1,v_2,v_3,v_4)$ be an undirected $5$-cycle of $\vec{\Gamma}$ and denote the vertices of $N[C]$ as in Agreement~\ref{okolicaC}. By way of contradiction suppose that for some (and thus each) vertex $v$ of $\G$ we have that $|G_v| > 2$. We consider the two cases separately.

Suppose first that $C$ is of type $3$. With no loss of generality assume that $(v_0 ,\, v_1 ,\, v_2,\,v_3)$ is the $3$-arc of $C$. By Lemma \ref{stabilizerSize} there exists an automorphism $\alpha \in G_{v_0}$ which maps the $2$-arc $(v_0 ,\, v_1 ,\, v_2)$ to $(v_0 ,\, v_4 ,\, v_3)$. Clearly, $\alpha$ maps $v_4$ to $v_1$. Since it cannot interchange adjacent vertices $v_2$ and $v_3$, $v^\alpha_3\ne v_2$ must hold. But then $(v_1 ,\, v^\alpha_3,\, v_3,\, v_2)$ is a $4$-cycle of $\G$, a contradiction. 

Suppose now that $C$ is of type $4$ and with no loss of generality assume that $(v_0 ,\, v_1 ,\, v_2,\, v_3,\,v_4)$ is the $4$-arc of $C$. By Lemma \ref{stabilizerSize} there exist automorphisms $\alpha ,\, \beta\in G_{v_0}$ both mapping $v_1$ to $v_4$ but such that $v^\alpha_2 = u_4$ while $ v^\beta_2 = w_4$. Since clearly $v^\alpha_4 = v^\beta_4 = v_1$, we also have that $v^\alpha_3=v^\beta_3$. But then $(w_4 ,\, v^\alpha_3,\, u_4,\, v_4)$ is a $4$-cycle of $\G$, a contradiction. 
\end{proof}

In each of the next two results the Rose window graph $\mathrm{R}_{12}(5,2)$ plays an important role. This graph is rather special not only from our viewpoint but also among all Rose window graphs themselves. Namely, using the classification of all edge-transitive Rose window graphs obtained by Kov\'acs, Kutnar and Maru\v si\v c~\cite{KovKutMar10} it can be verified that $\mathrm{R}_{12}(5,2)$ is the unique edge-transitive Rose window graph of girth $5$. Namely, with the terminology from~\cite{Wil08} the graphs from the first two families (as defined in~\cite{Wil08}) of edge-transitive Rose window graphs are easily seen to be of girth at most $4$, while the ones from the third family (we point out that in~\cite{KovKutMar10} this family is called {\em family (d)}) are bipartite, and so cannot be of girth $5$. Finally, it is not difficult to verify that the members of the fourth family are all of girth $6$, except for the two smallest examples $\mathrm{R}_{12}(1,4)$ and $\mathrm{R}_{12}(5,2)$, which are of girth $3$ and $5$, respectively.

\begin{proposition}
\label{no3}
Let $\Gamma$ be a tetravalent graph of girth $5$ admitting a half-arc-transitive subgroup $G\leq\text{Aut}(\Gamma )$ such that a corresponding $G$-induced orientation $\vec{\Gamma}$ gives rise to $5$-cycles of type $3$. Then $\Gamma \cong  \mathrm{R}_{12}(5,2)$.
\end{proposition}

\begin{proof}
Let $C = (v_0,\,v_1,\,v_2,\,v_3,\,v_4)$ be a $5$-cycle of type $3$ and denote the vertices of $N[C]$ as in Agreement~\ref{okolicaC}. With no loss of generality we can assume that $(v_0,\,v_1,\,v_2,\,v_3)$ is the $3$-arc of $C$. By Proposition \ref{z2} the stabilizer in $G$ of any vertex of $\Gamma$ is isomorphic to $\mathbb{Z}_2$. Proposition \ref{stabilizerSize} thus implies that $G$ has two orbits on the set of all $2$-arcs of $\vec{\Gamma}$, say $O_1$ and $O_2$, where $O_1$ contains the $2$-arc $(v_0,\, v_4,\, v_3)$. Note that $(v_0,\, v_1,\, v_2) \notin O_1$ as otherwise the unique involution of $G_{v_0}$ interchanges adjacent vertices $v_2$ and $v_3$. Similarly, $(v_1,\, v_2,\, v_3)\notin O_1$, as otherwise the unique involution of $G_{v_3}$ interchanges $v_0$ and $v_1$. Therefore, $(v_0,\, v_1,\, v_2), (v_1,\, v_2,\, v_3) \in O_2$. 

Suppose there exists a $5$-cycle $(x_0,\, x_1,\, x_2,\, x_3,\, x_4)$ of type $3$ with $(x_0,\, x_1,\, x_2),\, (x_1,\, x_2,\, x_3)\in O_1$ and $(x_0,\, x_4,\, x_3)\in O_2$. Then there also exists an $\alpha \in G$ sending $(v_0,\, v_1,\, v_2)$ to $(x_0,\, x_4,\, x_3)$ and consequently $(v_0,\, v_4,\, v_3)$ to $(x_0,\, x_1,\, x_2)$. But as $(v_2,v_3) \in A(\vG)$ while $(v^\alpha_2, \, v^\alpha_3) = (x_3,\, x_2) \notin A(\vG)$, this is impossible. Therefore, each $5$-cycle of type~3 has one $2$-arc from $O_1$ and two from $O_2$. It is thus clear that $G$ is transitive on the set of $5$-cycles of type $3$ and that every $2$-arc in $O_1$ lies on one $5$-cycle of type $3$ while every $2$-arc from $O_2$ lies on two $5$-cycles of type $3$ where it is the first $2$-arc of the $3$-arc of one of them and is the last $2$-arc of the $3$-arc of the other.

\begin{figure}[h]
    \centering
    \subfloat[\centering ]{{\includegraphics[height=4 cm]{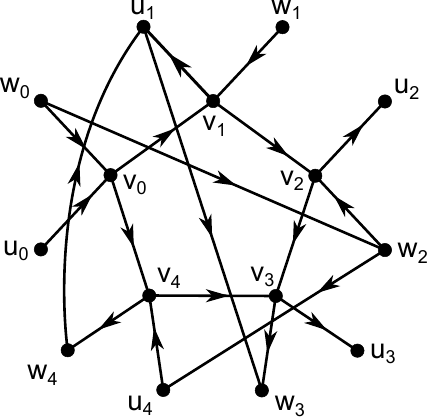} }}%
    \qquad
    \subfloat[\centering ]{{\includegraphics[height=5 cm]{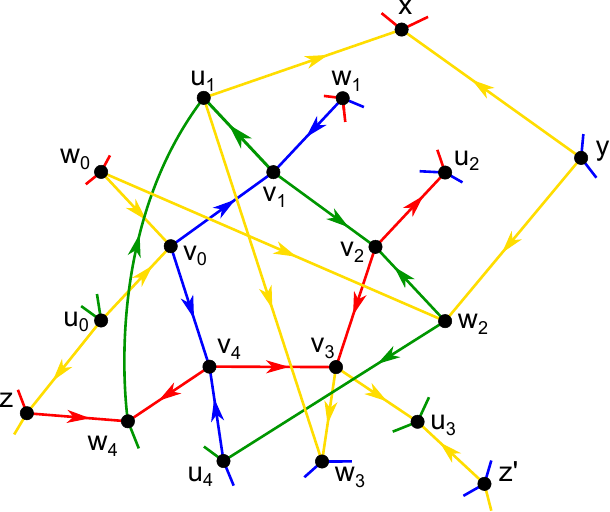} }}%
    \caption{$\vec{\Gamma}$ with $5$-cycles of type 3}%
    \label{Type3}%
\end{figure}

Since $(v_0,\, v_1,\, v_2)\in O_2$, Lemma~\ref{stabilizerSize} implies that the $2$-arc $(v_0,\, v_1,\, u_1)$ is in $O_1$, and so the only $5$-cycle it lies on must be $(v_0,\, v_1,\, u_1,\, w_4,\, v_4)$, forcing $(w_4,\, u_1)\in A(\vec{\Gamma} )$. Similarly, $(v_1,\, v_2,\, v_3)\in O_2$ yields $(w_2,\, v_2,\, v_3)\in O_1$, and so the unique $5$-cycle it lies on must be $(w_2,\, v_2,\, v_3,\, v_4,\, u_4)$, forcing $(w_2,\, u_4)\in A( \vec{\Gamma} )$. Notice that $(w_4,\, u_1,\, v_1,\, v_2,\, v_3,\, v_4)$ is a $6$-cycle whose edges have such orientation that we can conclude that every two $G$-alternating cycles with nonempty intersection must either share all vertices or exactly every third vertex (see~\cite[Proposition~2.6]{MTAO}). They cannot share all vertices since then \cite[Proposition~2.4.]{MTAO} implies that $\Gamma$ is a Cayley graph of a cyclic group and is thus of girth (at most) $4$. Every two $G$-alternating cycles with nonempty intersection therefore share exactly every third vertex (implying also that $\Gamma$ is not tightly-attached). 

The $2$-arc $(v_0,\, v_1,\, v_2)\in O_2$ is the first $2$-arc of the $3$-arc of $C$, and so there must also be a $5$-cycle $C'$ of type $3$ such that $(v_0,\, v_1,\, v_2)$ is the last $2$-arc of the $3$-arc of $C'$. Since $w_0$ and $u_0$ are both in-neighbors of $v_0$ and we have no restrictions on them thus far, we can assume we have the $5$-cycle $(v_0,\, v_1,\, v_2,\, w_2,\, w_0)$ with $(w_0,\, w_2)\in A(\vec{\Gamma} )$. Similarly, $(v_1,\, v_2,\, v_3)\in O_2$ must also be the first $2$-arc of the $3$-arc of some $5$-cycle of type $3$, and therefore we can assume $(u_1,\, w_3)\in A(\vec{\Gamma} )$. We thus have the part of $\vec{\Gamma}$ as in Figure \ref{Type3} (a). 

We now consider the $G$-alternating cycles of $\G$ (henceforth in this proof we omit the prefix $G$-) defined via the following table in which for each alternating cycle we set its name and give an alternating path it contains (note that this determines the alternating cycle completely):
$$
\begin{array}{c|c}
	\text{name} & \text{alternating path} \\
	\hline
	\mathcal{B} & (u_4,\, v_4,\, v_0,\, v_1,\, w_1) \\
	\mathcal{G} & (w_4,\,  u_1,\, v_1,\, v_2,\, w_2,\ u_4) \\
	\mathcal{R} & (w_4,\, v_4,\, v_3,\, v_2,\, u_2) \\
	\mathcal{Y} & (u_0,\, v_0,\, w_0,\, w_2)
\end{array}
$$
%We call the cycles $\mathcal{B}, \mathcal{G}, \mathcal{R}$ and $\mathcal{Y}$ {\em blue}, {\em green}, {\em red} and {\em yellow}, respectively.
%\vspace*{10pt}
%
%The \it Blue \rm alternating cycle $B$ includes the alternating path $[u_4,\, v_4,\, v_0,\, v_1,\, w_1]$.
%
%The \it Green \rm alternating cycle $Gr$ includes the alternating path $[w_4,\,  u_1,\, v_1,\, v_2,\, w_2,\ u_4]$.
%
%The \it Red \rm alternating cycle $R$ includes the alternating path $[w_4,\, v_4,\, v_3,\, v_2,\, u_2]$.
%
%The \it Yellow \rm alternating cycle $Y$ includes the alternating path $[u_0,\, v_0,\, w_0,\, w_2]$.
%
%\vspace*{10pt}
Observe first that since any two alternating cycles with nonempty intersection share every third vertex these four alternating cycles are pairwise distinct. Moreover, as $w_2\in V(\mathcal{G})\cap V(\mathcal{Y})$ and $u_1$ is at distance $3$ on $\mathcal{G}$ from it, we must have that $u_1\in V(\mathcal{Y})$ holds. Consequently, $\mathcal{Y}$ also includes the alternating path $(u_1,\, w_3,\, v_3,\, u_3)$. We thus see that any two of these four alternating cycles have a nonempty intersection. As we already know that each pair shares every third vertex, this in fact implies that $\mathcal{B}$, $\mathcal{G}$, $\mathcal{R}$ and $\mathcal{Y}$ are the only alternating cycles of $\G$. Moreover, $u_0, u_3 \in V(\mathcal{G})$, $w_0, w_1 \in V(\mathcal{R})$ and $u_2, w_3 \in V(\mathcal{B})$. 

The vertices $w_4$ and $v_2$ are at distance $3$ on both $\mathcal{G}$ and $\mathcal{R}$. Since $u_1$ and $w_2$ are at distance $3$ on $\mathcal{G}$, they must therefore also be at distance $3$ on $\mathcal{Y}$. It thus follows that there exists neighbors $x$ and $y$ of $u_1$ and $w_2$, respectively, such that $x \in V(\mathcal{Y}) \cap V(\mathcal{R})$, $y \in V(\mathcal{Y}) \cap V(\mathcal{B})$ and $(u_1,\,x,\,y,\,w_2)$ is an alternating path contained in $\mathcal{Y}$. Since $\G$ is of girth $5$ it is easy to see that $x$ and $y$ are not in $N[C]$, that is, they are ``new'' vertices. Note that $(u_0,\, u_3)\notin A(\vec{\Gamma} )$, as the length of $\mathcal{Y}$ would then be $10$, which is not divisible by $3$. None of the vertices of $N[C]$ not already included in the part of $\mathcal{Y}$ we know by now can be a vertex of $\mathcal{Y}$ since for each of them we already know on which two alternating cycles it lies. There must therefore be two ``new'' vertices, say $z \in V(\mathcal{Y}) \cap V(\mathcal{R})$ and $z' \in V(\mathcal{Y}) \cap V(\mathcal{B})$, where $(u_0,\,z), (z',u_3) \in A(\vG)$. 

To complete the proof note that the fact that $(w_0,\, v_0,\, v_1,\, v_2,\, w_2)$ is a $5$-cycle of type~3 implies that $(w_0,\,v_0,\,v_1) \in O_2$. Since we also have that $(v_0,\,v_4,\,v_3) \in O_1$, Lemma~\ref{stabilizerSize} implies that $(u_0,\,v_0,\,v_4), (v_0,\,v_4,\,w_4) \in O_2$, and so there must be a $5$-cycle of type~3 whose $3$-arc is $(u_0,\,v_0,\,v_4,\, w_4)$. This clearly forces $(z,\,w_4) \in A(\vG)$, and so we have the situation presented in Figure~\ref{Type3} (b). The vertices $z$ and $v_3$ are at distance $3$ on $\mathcal{R}$, and must therefore also be at distance $3$ on $\mathcal{Y}$. This finally forces $(z',\, z)\in A(\vec{\Gamma} )$, meaning that the alternating cycles are of length $12$ and consequently that the graph $\Gamma$ is of order $24$. Inspecting the census by Poto\v cnik and Wilson (see \cite{PotWil}) we find that $\Gamma$ must be $\mathrm{R}_{12}(5,2)$ (as it is the only graph of girth $5$ and order $24$ in the census).
\end{proof}

\begin{proposition}
\label{no4}
Let $\Gamma$ be a tetravalent graph of girth $5$ admitting a half-arc-transitive subgroup $G\leq\text{Aut}(\Gamma )$ such that a corresponding $G$-induced orientation $\vec{\Gamma}$ gives rise to $5$-cycles of type $4$. Then $\Gamma \cong \mathcal{X}_o(3,\, 9;\, 4)$ or $\Gamma \cong \mathrm{R}_{12}(5,2)$. 
\end{proposition}

\begin{proof}
By Theorem~\ref{TA} the only way $\G$ can be $G$-tightly-attached is if $\G \cong \mathcal{X}_o(3,\, 9;\, 4)$. For the rest of the proof we thus assume that $\Gamma$ is not $G$-tightly-attached. Let $C = (v_0,\,v_1,\,v_2,\,v_3,\,v_4)$ be a $5$-cycle of type $4$ and denote the vertices of $N[C]$ as in Agreement~\ref{okolicaC}, where with no loss of generality we assume that $(v_0,\,v_1,\,v_2,\,v_3,\,v_4)$ is the $4$-arc of $C$. By Proposition~\ref{z2} the stabilizer in $G$ of any vertex of $\Gamma$ is isomorphic to $\mathbb{Z}_2$. Proposition~\ref{stabilizerSize} thus implies that $G$ has two orbits on the set of all $2$-arcs of $\vec{\Gamma}$, say $O_1$ and $O_2$. Moreover, since each vertex of $\G$ is the middle vertex of exactly two $2$-arcs from each of $O_1$ and $O_2$, we have that $|O_1|=|O_2|=2|V|$. 

Denote the number of all $5$-cycles of type $4$ of $\G$ by $c$ and for each $i \in \{1,2\}$ let $t_i$ be the number of $5$-cycles of type $4$ on which each given $2$-arc from $O_i$ lies. Observe that each $5$-cycle of type $4$ is completely determined by specifying the last arc of its $4$-arc. Namely (using the above notation of vertices of $C$), since $(v_3,\, v_4,\, v_0,\, v_1)$ is an alternating path in $\vec{\G}$, it is completely determined by $(v_3,\, v_4)$ and then Lemma \ref{basicProperties}\ref{ii} implies that there is only one possibility for the vertex $v_2$. This shows that $c$ equals the number of arcs of $\vec{\Gamma}$, that is, $c = 2|V|$, and that all $5$-cycles of type $4$ are in the same $G$-orbit. Counting the number of pairs of a $2$-arc of $\vec{\Gamma}$ and a $5$-cycle of type~4 containing it in two different ways we find that $|O_1|\, t_1+|O_2|\, t_2 = 3c$, and so 
%$$
%	|O_1|\, t_1+|O_2|\, t_2 = 2 |V|\, (t_1+t_2)\quad \text{and}\quad |O_1|\, t_1+|O_2|\, t_2= 3c=6|V|.
%$$
$t_1+t_2 = 3$. %If one of $t_1$, $t_2$ is $0$ (and thus the other is $3$) then $(v_j,\,v_{j+1},\,v_{j+2})$, $0 \leq j \leq 2$, are all in the same orbit, and so there exists some $\alpha \in G$ mapping $(v_0,\,v_1,\,v_2)$ to $(v_1,\,v_2,\,v_3)$ and consequently $(v_1,\,v_2,\,v_3)$ to $(v_2,\,v_3,\,v_4)$. But this contradicts Lemma~\ref{basicProperties}\ref{ii} (consider $C$ and its $\alpha$-image). With no loss of generality we can thus assume that $t_1 = 1$ and $t_2 = 2$.

Suppose the $2$-arc $(v_1,\, v_2,\, v_3)$ lies on a $5$-cycle $C'$ of type $4$, different from $C$. Lemma \ref{basicProperties}\ref{ii} then implies that $C'$ contains one of $u_1$ and $w_1$ and one of $u_3$ and $w_3$. Since $C'$ is of type $4$, it contains at most one of $u_1$ and $w_3$. If it contains $u_1$, then $(u_3,\,u_1) \in A(\vG)$, and so $v_1$ and $u_3$ are on the same two $G$-alternating cycles (we omit the prefix $G$- henceforth in this proof) and have distance $2$ on one of them, contradicting the fact that $\G$ is not tightly-attached. Similarly, if $C'$ contains $w_3$, then $(w_3,\,w_1) \in A(\vG)$, and so $w_1$ and $v_3$ are on the same two alternating cycles having distance $2$ on one of them. Therefore, $C'$ contains $w_1$ and $u_3$ and $(w_1,\,u_3) \in A(\vG)$. But then the unique involution of $G_{v_1}$ fixes $v_3$, implying that $v_1$ and $v_3$ are antipodal vertices of a $4$-cycle, a contradiction. This shows that $(v_1,\, v_2,\, v_3)$ lies on a unique $5$-cycle of type $4$, proving that one of $t_1$ and $t_2$ is $1$ and the other is $2$. With no loss of generality we can assume $t_1 = 1$ and $t_2 = 2$. This also shows that $(v_0,\,v_1,\,v_2), (v_2,\,v_3,\,v_4) \in O_2$, since we can otherwise map one of these $2$-arcs to $(v_1,\,v_2,\,v_3)$, providing another $5$-cycle through this $2$-arc. Since there is just one $G$-orbit of $5$-cycles of type~4 this shows that for each $5$-cycle $C''$ of type $4$ the first and last $2$-arc of its $4$-arc (which we simply call the {\em first} and {\em last} $2$-arc of $C''$, respectively) are in $O_2$ while the middle one (we call it the {\em middle} $2$-arc of $C''$) is in $O_1$. 

With no loss of generality we can assume that $(u_0,\, v_0,\, v_1),\, (v_3,\, v_4,\, u_4)\in O_1$. Then the second $5$-cycle through the $2$-arc $(v_0,\, v_1 ,\, v_2 )$ and $(v_2,\, v_3 ,\, v_4)$, respectively, is $(v_0,\, v_1 ,\, v_2,\, w_2,\, u_0)$ and $(v_2,\, v_3 ,\, v_4,\, u_4,\, u_2)$, respectively, giving us $(w_2,\, u_0),\, (u_4,\, u_2)\in A(\vec{\Gamma} )$. Recall that each alternating path of length $3$ lies on a (unique) $5$-cycle of type $4$. Considering the alternating paths $(v_0,\, v_4,\, v_3,\, u_3)$ and $(v_4,\, v_0,\, v_1,\, w_1)$ and taking into account that by Lemma \ref{stabilizerSize} $(u_0,\, v_0,\, v_4)$, $(v_0,\, v_4,\, u_4)\in O_2$ we find that $(u_3,\, u_0),\, (u_4,\, w_1)\in A(\vec{\Gamma} )$, giving us the part of $\vec{\Gamma}$ in Figure \ref{type4} (a).

\begin{figure}[h]
    \centering
    \subfloat[\centering ]{{\includegraphics[height=3 cm]{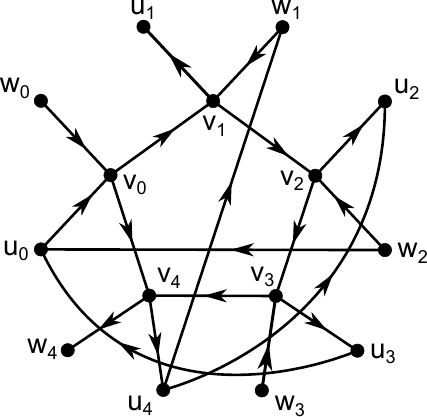} }}%
    \qquad
    \subfloat[\centering ]{{\includegraphics[height=4 cm]{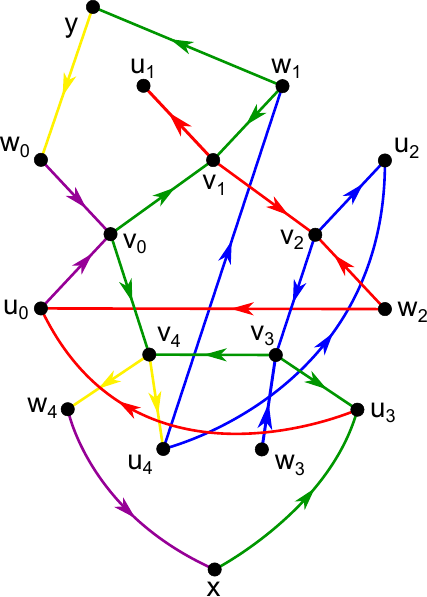} }}%
    \qquad
    \subfloat[\centering ]{{\includegraphics[height=4 cm]{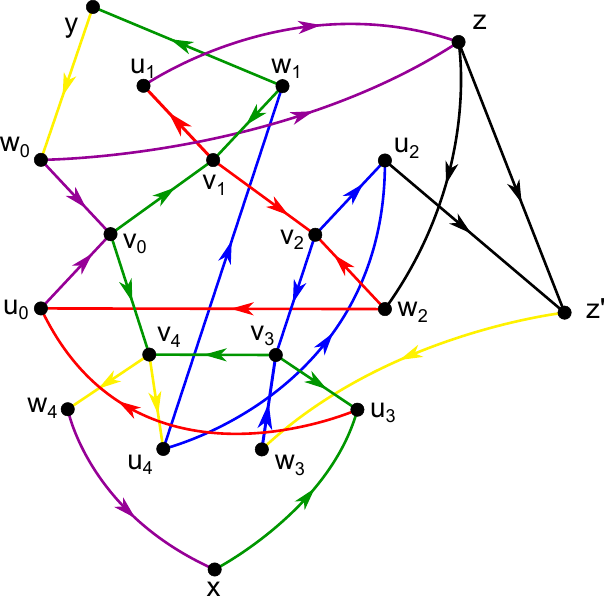} }}%
    \qquad
    \subfloat[\centering ]{{\includegraphics[height=4 cm]{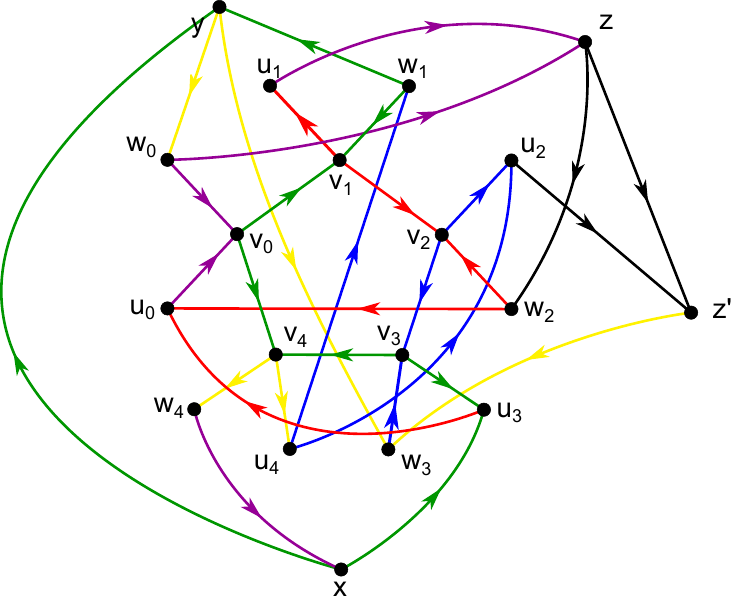} }}%
    \qquad
    \caption{$\vec{\Gamma}$ with $5$-cycles of type 4}%
    \label{type4}%
\end{figure}

This reveals that the vertices $w_1$ and $v_3$ lie on the same two alternating cycles and that their distance on each of them is a divisor of $4$. As $\Gamma$ is not tightly-attached we conclude that every two alternating cycles with nonempty intersection share exactly every fourth vertex. Consequently, the alternating cycles are of length at least $8$ and every alternating cycle has a nonempty intersection with precisely $4$ other alternating cycles. 

As in the proof of Proposition~\ref{no3} we now consider alternating cycles of $\G$ which we again describe via a  table in which for each alternating cycle we set its name and give an alternating path it contains:
$$
\begin{array}{c|c}
	\text{name} & \text{alternating path} \\
	\hline
	\mathcal{B} & (w_1,\, u_4,\, u_2,\, v_2,\, v_3,\, w_3) \\
	\mathcal{G} & (w_1,\,  v_1,\, v_0,\, v_4,\, v_3,\, u_3) \\
	\mathcal{R} & (u_1,\, v_1,\, v_2,\, w_2,\, u_0,\, u_3) \\
	\mathcal{Y} & (w_4,\, v_4,\, u_4) \\
	\mathcal{P} & (w_0,\, v_0,\, u_0)
\end{array}
$$
%We call the cycles $\mathcal{B}, \mathcal{G}, \mathcal{R}$, $\mathcal{Y}$ and $\mathcal{P}$ {\em blue}, {\em green}, {\em red}, {\em yellow} and {\em purple}, respectively.
%\vspace*{10pt}
%
%The \it Green \rm alternating cycle $Gr$ includes the alternating path $[w_1,\,  v_1,\, v_0,\, v_4,\, v_3,\, u_3]$.
%
%The \it Blue \rm alternating cycle $B$ includes the alternating path $[w_1,\, u_4,\, u_2,\, v_2,\, v_3,\, w_3]$.
%
%The \it Yellow \rm alternating cycle $Y$ includes the alternating path $[w_4,\, v_4,\, u_4]$.
%
%The \it Red \rm alternating cycle $R$ includes the alternating path $[u_1,\, v_1,\, v_2,\, w_2,\, u_0,\, u_3]$.
%
%The \it Purple \rm alternating cycle $P$ includes the alternating path $[w_0,\, v_0,\, u_0]$.

Since any two alternating cycles with nonempty intersection share every fourth vertex, these five alternating cycles are pairwise distinct. Let us consider how the cycle $\mathcal{G}$ extends beyond the above given alternating path. We cannot have $(w_1,\, u_3)\in A(\vec{\Gamma} )$ as $\mathcal{G}$ must be of length at least $8$. Moreover, none of the vertices of $N[C]$ which are not already in the above alternating path included in $\mathcal{G}$ can be a vertex of $\mathcal{G}$ since each of them has distance $1$ or $2$ to a different vertex of $\mathcal{G}$ via an alternating cycle different from $\mathcal{G}$. There are thus distinct vertices $x$ and $y$ outside $N[C]$ such that $(x,\, u_3)$ and $(w_1,\, y)\in A(\vec{\Gamma})$. Each of the alternating paths $(y,\, w_1,\, v_1,\, v_0)$ and $( v_4,\, v_3,\, u_3,\, x)$ must lie on a $5$-cycle of type $4$, and so the fact that $(w_0,\, v_0,\, v_1),\, (v_3,\, v_4,\, w_4)\in O_2$ forces $(y,\, w_0),\, (w_4,\, x)\in A(\vec{\Gamma} )$. Since $v_4 \in V(\mathcal{Y})$ and $y$ is at distance $4$ on $\mathcal{G}$ from $v_4$, it follows that $y \in V(\mathcal{Y})$ and similarly $x \in V(\mathcal{P})$. This gives us the part of $\vec{\Gamma}$ as in Figure \ref{type4} (b). 

Since $(w_0,\, v_0,\, v_1)\in O_2$, this $2$-arc must be the first $2$-arc of a $5$-cycle of type $4$ and since $(v_0,\, v_1,\, u_1)\in O_1$, this $5$-cycle must include $u_1$. There thus is a vertex $z$ of $\Gamma$ such that $(w_0,\, z),\, (u_1,\, z)\in A(\vec{\Gamma})$ and these two arcs are both contained in $\mathcal{P}$. Of the vertices already on Figure~\ref{type4}~(b) only $w_3$ is a potential candidate for $z$ (if $z = w_2$ we get a $4$-cycle). But as the distance between $w_3$ and $u_4$ on $\mathcal{B}$ is $4$ and $u_4\in V(\mathcal{Y})$ we have that $w_3\in V(\mathcal{Y})\cap V(\mathcal{B})$, showing that $z\ne w_3$. Thus, $z$ is a ``new'' vertex. 

As the $2$-arc $(v_1,\, u_1 ,\, z)\in O_2$ also is the first $2$-arc of some $5$-cycle of type $4$, we must have $(z,\, w_2 )\in A(\vec{\Gamma} )$. Similarly $(z,\, w_2 ,\,  v_2)\in O_2$ must be the first $2$-arc on some $5$-cycle of type $4$, and so as $(w_2,\,v_2,\,u_2) \in O_1$, $z$ and $u_2$ have a common out-neighbor, say $z'$. It can easily be shown that $z'$ must again be a ``new'' vertex. Since $(v_2,\, u_2 ,\,  z')\in O_2$ and $(w_3,\, v_3,\, v_2,\, u_2)$ is an alternating path, $(z',\, w_3)\in A( \vec{\Gamma} )$ holds. Recall that $w_3\in V(\mathcal{Y})$. This gives us the part of $\vec{\Gamma}$ as in Figure \ref{type4} (c).

To complete the proof note that as $(y,\, w_0,\, v_0) \in O_1$ the $2$-arc $(y,\, w_0,\, z)\in O_2$ is the first $2$-arc of a $5$-cycle of type $4$. Since $(w_0,\,z,\,w_2) \in O_2$ and $(u_2,\,z',\,w_3) \in O_1$, we have that $(w_0,\, z,\, z')\in O_1$ and $(z,\, z',\, w_3)\in O_2$, and so we must have $(y,\, w_3)\in A( \vec{\Gamma} )$. Lastly, since $(z',\, w_3,\, v_3) \in O_2$, we have that $(y,\, w_3,\, v_3)\in O_1$, and so the fact that $(w_3,\, v_3,\, u_3)\in O_2$ implies that the unique $5$-cycle of type $4$ containing $(y,\, w_3,\, v_3)$ is $(y,\, w_3,\, v_3,\, u_3,\, x)$, implying that $(x,\, y)\in A(\vec{\Gamma}) $. This means that $\mathcal{G}$, and consequently each alternating cycle, is of length $8$. Inspecting them on Figure \ref{type4} (d) we notice that the alternating path $(w_2,\, z,\, z',\, u_2)$ must be a part of a ``new'' alternating cycle. We now see that each of these six alternating cycles shares vertices with four other ones that we already have. This means that $\Gamma$ has precisely $6$ alternating cycles of length $8$ and consequently is of order $24$. As in the proof of Proposition \ref{no3} this implies that $\Gamma$ is $\mathrm{R}_{12}(5,2)$.
\end{proof}

That the Rose window graph $\G = \mathrm{R}_{12}(5,2)$ does indeed admit half-arc-transitive groups $G_1$ and $G_2$ giving rise to $5$-cycles of type~3 and~4, respectively, can be verified by {\sc Magma}~\cite{magma}. But one can also use the results of~\cite{RoseWindow} where the full automorphism group of each edge-transitive Rose window graph was determined. In particular,~\cite[Proposition 3.5.]{RoseWindow} implies that $\Aut(\G) = \langle \rho,\, \mu ,\, \sigma \rangle$, where 
\begin{align*}
\rho & =(x_0,\, x_1,\,\ldots ,\, x_{11})(y_0,\, y_1,\, \ldots ,\, y_{11}),\\
\mu & = \Pi_{i=0}^{11} (x_i,\, x_{-i})(y_i,\, y_{-i-5})\ \text{and}\\
\sigma & = (x_1\, y_0)(x_2\, y_{10})(x_4\, y_3)(x_5\, y_1)(x_7\, y_6)(x_8\, y_4)(x_{10} \, y_9)(x_{11}\, y_7)(y_2\, y_8)(y_5\, y_{11}),
\end{align*}
and that $\Aut(\G)$ acts regularly on the arc set $A(\G)$. It is clear that both $G_1 = \langle\rho,\, \sigma \rangle$  and $G_2 = \langle \rho^2,\, \sigma\rho ,\, \mu \rangle$ are vertex- and edge-transitive (note that $\langle \rho^2 ,\, \mu \rangle$ has just one orbit on the set $\{y_i \colon i \in \ZZ_{12}\}$). It is easy to verify that $\{x_0,x_6,y_2,y_8\}$ is an imprimitivity block for $G_1$, and so the fact that $\mu$ preserves $x_0$ but maps $y_2$ to $y_5$ shows that $G_1$ is a proper subgroup of $\Aut(\G)$, implying that it is half-arc-transitive. Similarly, one can show that $\{x_0,x_4,x_8\}$ is an imprimitivity block for $G_2$, and so the fact that $\sigma$ preserves $x_0$ but maps $x_4$ to $y_3$ shows that $G_2$ is also half-arc-transitive. Note that this also shows that $G_1 \neq G_2$ (since $\sigma \in G_1$). For each of $G_1$ and $G_2$ one of the induced orientations of the edges of $\G$ is presented in Figure~\ref{RW12}, which reveals that the $5$-cycles are of type~4 for $G_1$ and are of type $3$ for $G_2$. 
\begin{figure}[h]
    \centering
\includegraphics[height=6 cm]{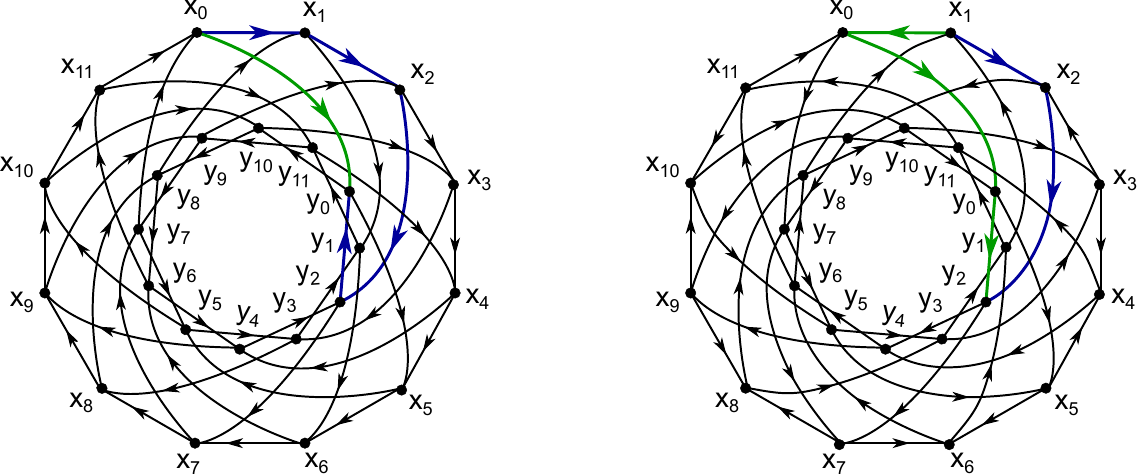}
    \caption{Orientations of the Rose window graph $\mathrm{R}_{12}(5,2)$ giving rise to $5$-cycles of type $4$ (left) and of type $3$ (right). }%
    \label{RW12}%
\end{figure}
Observe that the two representations of $\mathrm{R}_{12}(5,2)$ in Figure~\ref{RW12} also show that in the $G_1$-induced orientation of the edges all $5$-cycles are of type $4$, while in the $G_2$-induced orientation they are all of type~3. Combining this with Theorem~\ref{TA}, Proposition~\ref{no3} and Proposition~\ref{no4} thus gives a complete classification of the tetravalent graphs of girth $5$ admitting a half-arc-transitive group action giving rise to undirected $5$-cycles. 

%Interestingly $R^+$ is in one case a universal relation and in one case it is not. 

\begin{theorem}
\label{neusmerjeni}
A tetravalent graph $\Gamma$ of girth $5$ admits a half-arc-transitive subgroup $G\leq\text{Aut}(\Gamma )$ giving rise to undirected $5$-cycles if and only if $\Gamma \cong \mathcal{X}_o(3,\, 9;\, 4)$ or $\Gamma \cong \mathrm{R}_{12}(5,2)$. 
\end{theorem}

The above theorem gives another reason why the Doyle-Holt graph is special. Not only is it the smallest half-arc-transitive graph but it is also the unique tetravalent half-arc-transitive graph of girth $5$ where in any of the two orientations of the edges induced by the action of the automorphism group the $5$-cycles are not directed. 

%%%%%%%%%%%%%%%%%%%%%%%%%%%%%%%%%%%%%%%%%%%%%%%%%

\end{section}

\begin{section}{$G$-half-arc-transitive graphs with directed $5$-cycles}
\label{sec:directed}

In this section we consider tetravalent graphs of girth $5$ admitting a half-arc-transitive group action such that a corresponding orientation of the edges gives rise to directed $5$-cycles. Throughout this section let $\G$ be such a graph, $G \leq \Aut(\G)$ be such a half-arc-transitive subgroup and $\vG$ be a corresponding orientation. By Theorem~\ref{neusmerjeni} and the comments preceding it $\G$ has no undirected $5$-cycles, and so all of its $5$-cycles are directed. Lemma~\ref{basicProperties}\ref{ii} thus implies that each $2$-arc of $\vec{\Gamma}$ lies on at most two $5$-cycles. We first show that it in fact lies on at most one.

\begin{proposition}\label{2arc5c}
Let $\Gamma$ be a tetravalent graph of girth $5$ admitting a half-arc-transitive subgroup $G\leq\text{Aut}(\Gamma )$ such that a corresponding $G$-induced orientation $\vec{\Gamma}$ gives rise to directed $5$-cycles. Then a $2$-arc of  $\vec{\Gamma}$ lies on at most one $5$-cycle.
\end{proposition}

\begin{proof}
By way of contradiction suppose there exists a $2$-arc of $\vec{\Gamma}$ which lies on two $5$-cycles. 
%There are two possibilities. Either all $2$-arcs of $\vec{\Gamma}$ lie on two $5$-cycles (and there is either one or there are two $G$-orbits of $2$-arcs of $\vec{\Gamma}$) or there are two $G$-orbits of $2$-arcs of $\vec{\Gamma}$ with $2$-arcs of one $G$-orbit laying on one $5$-cycle and the ones in the other $G$-orbit laying on two $5$-cycles. 
We distinguish two possibilities depending on whether all $2$-arcs of $\vG$ lie on two $5$-cycles or not.

Suppose first that some $2$-arcs of $\vG$ lie on one $5$-cycle. Lemma~\ref{stabilizerSize} then implies that $G$ has two orbits on the set of all $2$-arcs of $\vG$, say $O_1$ and $O_2$, where the ones in $O_1$ lie on one $5$-cycle of $\G$ and those in $O_2$ lie on two. Moreover, $|G_v| = 2$ for all $v \in V(\G)$, and so the stabilizer in $G$ of an arc of $\vG$ is trivial. For any two $2$-arcs of $\vec{\Gamma}$ from the same $G$-orbit there is thus a unique element of $G$ mapping one to the other. 

Let $(v_0,\, v_1,\, v_2)\in O_1$ and let $C=(v_0,\, v_1,\, v_2,\, v_3,\, v_4)$ be the unique $5$-cycle of $\Gamma$ containing $(v_0,\, v_1,\, v_2)$. Denote the vertices of $N[C]$ as in Agreement~\ref{okolicaC} and recall that the orientations of the edges of $N[C]$ are as in Figure~\ref{figure3}. As $(v_0,\, v_1,\, v_2)\in O_1$, Lemma~\ref{stabilizerSize} implies that $(v_0,\, v_1,\, u_1),\, (w_1,\, v_1,\, v_2)\in O_2$. By Lemma~\ref{basicProperties}\ref{ii} there thus exist $5$-cycles in $\Gamma$, one going through $(v_4,\, v_0,\, v_1,\, u_1)$ and one through $(w_1,\, v_1,\, v_2,\, v_3)$, implying that $(v_4,\, v_0,\, v_1),\, (v_1,\, v_2,\, v_3)\in O_2$. Suppose now that $(v_i,\, v_{i+1},\, v_{i+2})\in O_1$ for some $i\in\mathbb{Z}_5$. The unique element of $G$ mapping $(v_0,\, v_1,\, v_2)$ to $(v_i,\, v_{i+1},\, v_{i+2})$ then fixes $C$ set-wise, and so as $C$ is of prime length $5$ and $(v_4,\, v_0,\, v_1) \in O_2$, $i=0$ must hold. This shows that each $5$-cycle has at most one $2$-arc from $O_1$. Lemma~\ref{stabilizerSize} thus implies that $(w_2,\, v_2, \, v_3), (v_3,\, v_4,\, u_4) \in O_1$. However, as $(v_2,\, v_3,\, v_4)\in O_2$, the other $5$-cycle going through it is $(w_2,\, v_2,\, v_3,\, v_4,\, u_4)$, contradicting the fact that no $5$-cycle contains more than one $2$-arc from $O_1$. 

We are left with the possibility that all $2$-arcs of $\vec{\Gamma}$ lie on two $5$-cycles. Again, let $C$ and $N[C]$ be as in Agreement~\ref{okolicaC}. Lemma~\ref{basicProperties}\ref{ii} implies that every $3$-arc of $\vec{\Gamma}$ lies on precisely one $5$-cycle. Considering the $3$-arcs $(w_i,\, v_i,\, v_{i+1},\, v_{i+2})$ for $i\in\mathbb{Z}_5$, Lemma \ref{basicProperties}\ref{ii} thus forces $(u_{i+2},\, w_{i})\in A(\vec{\Gamma})$ for all $i\in \mathbb{Z}_5$. The $3$-arc $(w_0,\, v_0,\, v_1,\, u_1)$ also lies on a $5$-cycle, and so there exists a vertex $x$ such that $(u_1,\, x),\, (x,\, w_0)\in A(\vec{\Gamma} )$. Since $\Gamma$ is of girth $5$, this $x$ cannot be any of $u_i,\, w_i$ with $i\ne 3$, while as there are no undirected $5$-cycles, $x$ is none of $u_3$ and $w_3$. Thus, $x$ is outside $N[C]$. In a similar way the fact that $(u_1,\, w_4,\, v_4,\, u_4)$ and $(u_4,\, w_2,\, v_2,\, u_2)$ are $3$-arcs of $\vec{\Gamma}$ implies that there are distinct $y,\, z$ outside $N(C)$ such that $(u_4,\, y),\, (y,\, u_1),\, (u_2,\, z),\, (z,\, u_4)\in A(\vec{\Gamma} )$ (see Figure \ref{figure3}). As the $3$-arc $(u_4,\, y,\, u_1,\, x)$ of $\vG$ must also lie on a $5$-cycle and $(z,\, u_4),\, (v_4,\, u_4)\in A(\vec{\Gamma} )$, we see that $(x,\, z)\in A(\vec{\Gamma})$ holds. But this forces the $4$-cycle $(x,\, z,\, u_2,\, w_0)$, a contradiction. Therefore, a $2$-arc of $\vec{\Gamma}$ lies on at most one $5$-cycle, as claimed. 
\begin{figure}[h]
    \centering
\includegraphics[height=4 cm]{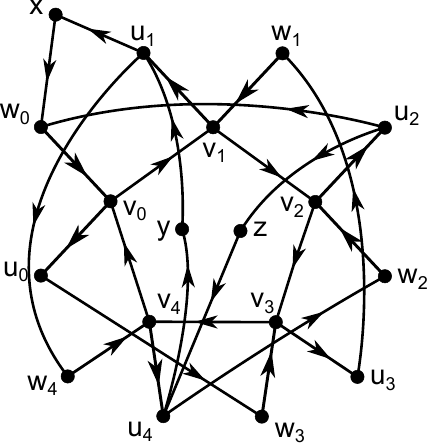}
    \caption{The situation in the proof of Proposition \ref{2arc5c} .}%
    \label{figure3}%
\end{figure}
\end{proof}

\begin{corollary}\label{stab4}
Let $\Gamma$ be a tetravalent graph of girth $5$ admitting a half-arc-transitive subgroup $G\leq\text{Aut}(\Gamma )$ such that a corresponding $G$-induced orientation $\vec{\Gamma}$ gives rise to directed $5$-cycles. Then each edge of $\G$ lies on at most two $5$-cycles. Moreover, $|G_v| \leq 4$ holds for each $v \in V(\G)$ and if $|G_v| = 4$ then each edge of $\G$ lies on two $5$-cycles of $\G$.
\end{corollary}

\begin{proof}
Since all $5$-cycles of $\G$ are directed, Proposition~\ref{2arc5c} implies that there can be at most two $5$-cycles through any given edge of $\G$. Since this also implies that some $3$-arcs of $\vec{\Gamma}$ do not lie on $5$-cycles, we also see that $|G_v| \leq 4$ for each $v \in V(\G)$. The last claim follows from Lemma~\ref{stabilizerSize}. 
\end{proof}

As was mentioned in the Introduction, a half-arc-transitive subgroup $G \leq \Aut(\G)$ for a tetravalent graph $\G$ has two orbits of $G$-consistent cycles. While in principle it can happen that the girth cycles of $\G$ are not $G$-consistent (like for instance in the case of the Doyle-Holt graph) we now show that at least in the case of girth $5$ graphs the only two examples in which the girth cycles are not $G$-consistent are the two graphs with undirected $5$-cycles (the Doyle-Holt graph $\mathcal{X}_o (3,\,9;\,4)$ and the Rose window graph $\mathrm{R}_{12}(5,2)$).

\begin{proposition}\label{consistent}
Let $\Gamma$ be a tetravalent graph of girth $5$ admitting a half-arc-transitive subgroup $G\leq\text{Aut}(\Gamma)$ and suppose $\G$ is not the Doyle-Holt graph $\mathcal{X}_o (3,\,9;\,4)$ or the Rose window graph $\mathrm{R}_{12}(5,2)$. Then the $5$-cycles of $\G$ are all $G$-consistent. 
\end{proposition}

\begin{proof}
Theorem~\ref{neusmerjeni} implies that the $5$-cycles of $\G$ are all directed (with respect to a $G$-induced orientation $\vG$). Let $C = (v_0,\, v_1,\, v_2,\, v_3,\, v_4)$ be a $5$-cycle of $\G$. By Corollary~\ref{stab4} each edge of $\G$ lies on at most two $5$-cycles of $\G$. If each edge lies on just one, then an element of $G$ taking $(v_0,\, v_1)$ to $(v_1,\, v_2)$ must be a shunt of $C$ (as $C$ is the unique $5$-cycle through $(v_1,\, v_2)$), showing that $C$ is $G$-consistent. Suppose thus that each edge of $\G$ lies on two $5$-cycles. Then Proposition~\ref{2arc5c} implies that every $2$-arc of $\vec{\Gamma}$ lies on precisely one $5$-cycle. As $C$ is of odd length, there exist two consecutive $2$-arcs of $C$ from the same $G$-orbit of $2$-arcs of $\vG$. But then an element of $G$ mapping the first to the second preserves $C$ and is thus its shunt, again showing that $C$ is $G$-consistent. 
\end{proof}

By Corollary~\ref{stab4} each edge of a tetravalent graph $\G$ of girth $5$ admitting a half-arc-transitive group action can lie on at most two directed $5$-cycles. We now show that if $\G$ is a Cayley graph it in fact cannot lie on two (recall that the {\em Cayley graph} $\text{Cay}(H,\, S)$ of the group $H$ with respect to its subset $S$, where $S = S^{-1}$ and $1 \notin S$, is the graph with vertex set $H$ in which each $h \in H$ is adjacent to all elements of the form $sh$, $s \in S$).  

\begin{proposition}\label{cay}
Let $H = \langle a, b \rangle$ be a group such that $S = \{a,\, a^{-1},\, b,\, b^{-1}\}$ is of cardinality $4$. Suppose that $\Gamma = \text{Cay}(H,\, S)$ is of girth $5$ and admits a half-arc-transitive subgroup $G\leq\text{Aut}(\Gamma )$. Then either $\Gamma$ is one of $\mathcal{X}_o(3,\, 9;\, 4)$ and $\mathrm{R}_{12}(5,2)$, or $a$ and $b$ are both of order $5$ and every edge of $\Gamma$ lies on a unique $5$-cycle of $\Gamma$.  
\end{proposition}

\begin{proof}
Let $\vec{\Gamma}$ be one of the two $G$-induced orientations of $\Gamma$. By Theorem~\ref{neusmerjeni} it suffices to assume that all $5$-cycles of $\G$ are directed with respect to $\vec{\Gamma}$ and then prove that $a$ and $b$ are of order $5$ and that every edge of $\Gamma$ lies on a unique $5$-cycle of $\G$. 

Since $\Gamma$ is a vertex-transitive graph of girth $5$, every vertex lies on a $5$-cycle. There thus exists a sequence $[s_4,\, s_3,\, s_2, \, s_1,\, s_0]$ with $s_i\in S$ and $s_{i+1}s_i\neq 1$ for all $i\in \mathbb{Z}_5$ (indices computed modulo $5$) such that $s_4s_3s_2s_1s_0=1$. Therefore, $(h,\, s_0h,\, s_1s_0h,\, s_2s_1s_0h,\, s_3s_2s_1s_0h)$ is a $5$-cycle for each $h\in H$. Since $ s_{i+1} \neq s_i^{-1}$ for each $i\in \mathbb{Z}_5$ and the sequence is of odd length, we must have that $s_{i+1}=s_i$ for at least one $i\in \mathbb{Z}_5$. We can assume that this holds for $i=0$ and that in fact $s_0=s_{1} = a$. 

Now, suppose that at least one of $s_i$, $i\in \mathbb{Z}_5$, is not $a$. Since by Proposition~\ref{2arc5c} the $2$-arc $(1,\, a,\, a^2)$ lies on at most one $5$-cycle of $\G$ it thus follows that $s_2,\, s_4\in \{b,\, b^{-1}\}$. Corollary~\ref{stab4} thus implies that there are precisely two $5$-cycles through $(1,\, a)$ (one containing $a^2$ and one $s_2a$). Note also that $s_3\in \{b,\, b^{-1}\}$ would imply $s_2=s_3=s_4$, contradicting the fact that the $2$-arc $(1,\, b,\, b^2)$ also lies on at most one $5$-cycle. Therefore, $s_3\in \{a,\, a^{-1}\}$. If $s_3=a^{-1}$, then $(1,\, a,\, s^{-1}_2a)$ also lies on a $5$-cycle (note that $s^{-1}_0 s^{-1}_1 s^{-1}_2 s^{-1}_3 s^{-1}_4=1$), contradicting the fact that there are just two $5$-cycles through $(1,\, a)$. Thus $s_3 = a$, and so $s_2=s_4$ must hold (otherwise we again get a $5$-cycle through $(1,\,a,\,s_2^{-1}a)$). But then the $2$-arc $(1,\, s_2,\, as_2)$ lies on two $5$-cycles, contradicting Proposition~\ref{2arc5c}. Thus $a^5=1$ (and similarly $b^5=1$) holds and each edge is contained in a unique $5$-cycle of $\Gamma$, as claimed.  
\end{proof}

We already know that there are infinitely many examples of tetravalent graphs of girth $5$ admitting a half-arc-transitive group action such that each edge of the graph lies on a unique $5$-cycle. Namely, the argument from Section~\ref{sec:TA} shows that each $\mathcal{X}_o(5,r;q)$ with $r$ odd and $q \in \ZZ_r$ such that $q \neq \pm 1$ and $1+q + q^2 + q^3 + q^4 = 0$ has this property. This also follows from Proposition~\ref{cay} and the results of~\cite{MTAO}, where one can see that the subgroup $\langle \rho, \sigma \rangle$, where $\rho$ and $\sigma$ are as in~\cite{MTAO}, is a regular group, and so $\mathcal{X}_o$ graphs are Cayley graphs.  

It turns out that there are also infinitely many examples of tetravalent graphs of girth $5$ admitting a half-arc-transitive group action such that each edge of the graph lies on two $5$-cycles, thus showing that Corollary~\ref{stab4} is best possible. We construct them as certain double coset graphs with respect to the groups $\PSL(2,p)$, where $p$ is a prime with $p \equiv 1 \pmod{10}$. Recall that $\PSL(2,p) = \mathrm{SL}(2,p)/\mathrm{Z}$, where $\mathrm{Z} = \{I, -I\}$ and $I$ is the identity matrix, and so we can view the elements of $\PSL(2,p)$ as matrices from $\mathrm{SL}(2,p)$ with the understanding that for each $A \in \mathrm{SL}(2,p)$ the matrices $A$ and $-A$ represent the same element of $\PSL(2,p)$. We start with a preparatory lemma.

\begin{lema}\label{matrike}
Let $p$ be a prime with $p \equiv 1 \pmod{10}$, let $\xi \in \ZZ_p$ be a fifth root of unity and let 
$$
	A = \begin{pmatrix} 0 & 1 \\ -1 & 0 \end{pmatrix}\quad \text{and}\quad B=\begin{pmatrix} \xi & \xi+\xi^{-1} \\ 0 & \xi^{-1} \end{pmatrix}
$$
be viewed as elements of $\PSL(2,p)$. Then $A$ is an involution, $B$ and $AB$ are both of order $5$ and $\langle A, B \rangle = \PSL(2,p)$.
\end{lema}

\begin{proof}
A straightforward calculation shows that $A$ is an involution (in $\PSL(2,p)$) and $B$ is of order $5$ (recall that $\xi$ is of order $5$ in the multiplicative group $\ZZ_p^*$, and so $1+\xi+\xi^2+\xi^3+\xi^4 = 0$). Moreover, one can verify that $(AB)^5 = -I$ (in $\mathrm{SL}(2,p)$), and so $AB$ is also of order $5$ in $\PSL(2,p)$. Letting $G = \langle A, B \rangle$ it thus remains to be verified that $G = \PSL(2,p)$. 
Observe first that 
$$
	ABA = \begin{pmatrix} -\xi^{-1} & 0 \\ \xi + \xi^{-1} & -\xi \end{pmatrix} \notin \langle B \rangle,
$$
showing that $G$ is neither abelian nor dihedral. Since $G$ contains an element of order $5$, the theorem of Dickson on subgroups of the $\PSL$ groups (see~\cite[Hauptsats 8.27]{dickson}) implies that if $G$ is not the whole $\PSL(2,p)$ then it must be isomorphic to the alternating group $A_5$ or to a semidirect product of the form $\ZZ_p \rtimes \ZZ_t$ for a divisor $t$ of $p-1$. 

It can be verified that in $A_5$ for each pair of an involution $a$ and an element $b$ of order $5$ such that $ab$ is also of order $5$ the element $ab^2$ is necessarily of order $3$. Calculating we find that 
$$
	AB^2 = \begin{pmatrix} 0 & \xi^{-2} \\ -\xi^2 & \xi+\xi^{-1}-1 \end{pmatrix} \quad \text{and}\quad
	(AB^2)^3 = \begin{pmatrix} 1-\xi-\xi^{-1} & -\xi^{2}(3-\xi+3\xi^2) \\ \xi(3-\xi+3\xi^2) & -3 + 6\xi + 6\xi^{-1} \end{pmatrix},
$$
and so $AB^2$ is not of order $3$ as $1-\xi-\xi^{-1} \in \{-1,1\}$ contradicts the assumption that $\xi$ has multiplicative order $5$. Therefore, $G$ is not isomorphic to $A_5$.

Finally, suppose that $G$ is isomorphic to a semidirect product of the form $\ZZ_p \rtimes \ZZ_t$ for a divisor $t$ of $p-1$. Then $G = \langle x,y \mid x^t = y^p = 1, x^{-1}yx = y^k \rangle$ for some $k \in \ZZ_p$ with $k^t = 1$. Since $G$ contains the involution $A$, $G$ is of even order, and so $t$ is even. Moreover, as $y$ is of odd order, $A$ is of the form $x^{t/2}y^j$, where $j \in \ZZ_p$. But as $B = x^iy^\ell$ for some $i \in \ZZ_t$, $i \neq 0$, and $\ell \in \ZZ_p$, we see that $B$ and $AB$ cannot both be of order $5$ (as this would force $t \mid 5i$ and $t \mid 5(i+t/2)$, which clearly cannot both hold). This finally proves that $G = \PSL(2,p)$, as claimed.
\end{proof}

In the proof of the next result we use a well-known construction of double coset graphs. For self containment we recall the definition. For a group $G$, its subgroup $H$ and a subset $S$ of $G$ the {\em double coset graph} $\mathrm{Cos}(G,H,S)$ of $G$ {\em with respect to} $H$ and $S$ is the graph whose vertex set is the set of all right cosets $Hg$ of $H$ in $G$ and in which the vertices $Hg$ and $Hg'$ are adjacent if and only if $g'g^{-1} \in HSH$. It is well known and easy to see that $G$ has a vertex-transitive action on $\mathrm{Cos}(G,H,S)$ via right multiplication which is faithful if and only if $H$ is core-free in $G$. 

\begin{proposition}
\label{pro:PSL}
Let $p$ be a prime with $p \equiv 1 \pmod{10}$. Then there exists a connected tetravalent arc-transitive graph of girth $5$ and order $p(p^2-1)/4$ admitting the group $\PSL(2,p)$ as a half-arc-transitive subgroup of automorphisms and each edge of this graph lies on two $5$-cycles.
\end{proposition}

\begin{proof}
Let $G = \PSL(2,p)$, let $A, B \in G$ be as in Lemma~\ref{matrike} and let $H = \langle A \rangle$. By Lemma~\ref{matrike} we have that $G = \langle A, B\rangle$, and so the double coset graph $\G = \mathrm{Cos}(G,H,\{B, B^{-1}\})$ is a connected graph of order $|G|/2 = p(p^2-1)/4$. As explained in the paragraph preceding this proposition $G$ has a faithful vertex-transitive action on $\G$. For better readability of the remainder of this proof we write $a$ and $b$ instead of $A$ and $B$, respectively. 

Observe first that the neighbors of the vertex $H$ are $Hb$, $Hb^{-1}$, $Hba$ and $Hb^{-1}a$. Since $a$ is an involution, $a$ and $b$ do not commute, and $ab$ and $b$ are both of order $5$ (see Lemma~\ref{matrike}), these four cosets are pairwise distinct, and so $\G$ is tetravalent. In addition, right multiplication by $a$ fixes $H$ and interchanges $Hb$ with $Hba$ and at the same time $Hb^{-1}$ with $Hb^{-1}a$, while right multiplication by $b$ maps the edge $\{Hb^{-1},H\}$ to the edge $\{H,Hb\}$. This shows that the action of $G$ on $\G$ is edge-transitive and is thus in fact half-arc-transitive (since the order of $G$ is twice the order of $\G$ it cannot be arc-transitive). Choose the $G$-induced orientation $\vG$ of $\G$ in which $Hb$ and $Hba$ are out-neighbors and $Hb^{-1}$ and $Hb^{-1}a$ are in-neighbors of $H$. Since right multiplication by $b$ is a shunt of a $5$-cycle through $(Hb^{-1},H,Hb)$ while right multiplication by $ab$ is a shunt of a $5$-cycle through $(Hb^{-1}a,H,Hb)$, there are thus two distinct directed $5$-cycles (with respect to $\vG$ - henceforth we simply speak of directed cycles) through the edge $\{H,Hb\}$.

We next prove that $\G$ is in fact of girth $5$. Since a $3$-cycle of $\G$ can only be a directed cycle (recall that no element of $G$ can swap a pair of adjacent vertices) and each of the $2$-arcs $(Hb^{-1},H,Hb), (Hb^{-1}a,H,Hb)$ of $\vG$ lies on a $G$-consistent $5$-cycle, it is clear that there are no $3$-cycles in $\G$. Since $(Hb^{-1}aba,Hba,H,Hb,Hb^{-1}ab)$ is a part of a $G$-alternating cycle, the $G$-alternating cycles are of length $4$ if and only if $Hb^{-1}aba = Hb^{-1}ab$, which holds if and only if $b^{-1}ab$ and $a$ commute. A straightforward calculation shows that 
$$
	b^{-1}aba = \begin{pmatrix} -2-\xi^2-2\xi^3 & 1+\xi^{2} \\ 1+\xi^2 & -\xi^{2} \end{pmatrix} \ \text{and}\ 
	\begin{pmatrix} -\xi^2 & -1-\xi^{2} \\ -1-\xi^2 & -2-\xi^2-2\xi^3 \end{pmatrix} = ab^{-1}ab.
$$
Note that $-1-\xi^2 \neq 1 + \xi^2$ (since $p$ is odd and $\xi^2 \neq -1$), and so the comment from the paragraph preceding Lemma 6.5 implies that these two elements of $\mathrm{PSL(2,p)}$ can only be equal if $-2-\xi^2-2\xi^3 = -(-\xi^2) = \xi^2$. But since $1+\xi+\xi^2+\xi^3+\xi^4 = 0$ and $p$ is odd, this would force $\xi = -\xi^{-1}$, contradicting the fact that $\xi$ is of multipli\-ca\-tive order $5$. Therefore, the $G$-alternating cycles of $\G$ are not $4$-cycles. The results of~\cite{HTG4} (in particular, the proof of~\cite[Theorem~4.1]{HTG4}) thus reveal that the only way $\G$ can be of girth $4$ is if it has directed $4$-cycles (it is easy to see that the so called ``exceptional'' graphs from~\cite{HTG4} are either bipartite or can only have directed $5$-cycles if they are of order at most $25$). Suppose there is a directed $4$-cycle through the edge $\{H,Hb\}$. It can share at most two consecutive edges with any of the above mentioned $5$-cycles through this edge given by the shunts $b$ and $ab$, and so one of $Hbab \to Hb^{-1}$ and $Hb^2 \to Hb^{-1}a$ must hold (see Figure~\ref{fig:PSL}). In the latter case the $2$-arc $(H,Hb,Hb^2)$ lies on a directed $4$-cycle, and so applying $b^{-1}$ we see that $(Hb^{-1},H,Hb)$ also lies on a directed $4$-cycle, showing that $Hbab \to Hb^{-1}$ does hold in any case. But then $Hb^2ab = Hb^{-1}$, which is impossible as $b^2ab^2 = 1$ yields $a = b$, while $b^2ab^2 = a$ yields $aba = b^{-1}$, contradicting Lemma~\ref{matrike}.
\begin{figure}[h]
    \centering
\includegraphics[height=4 cm]{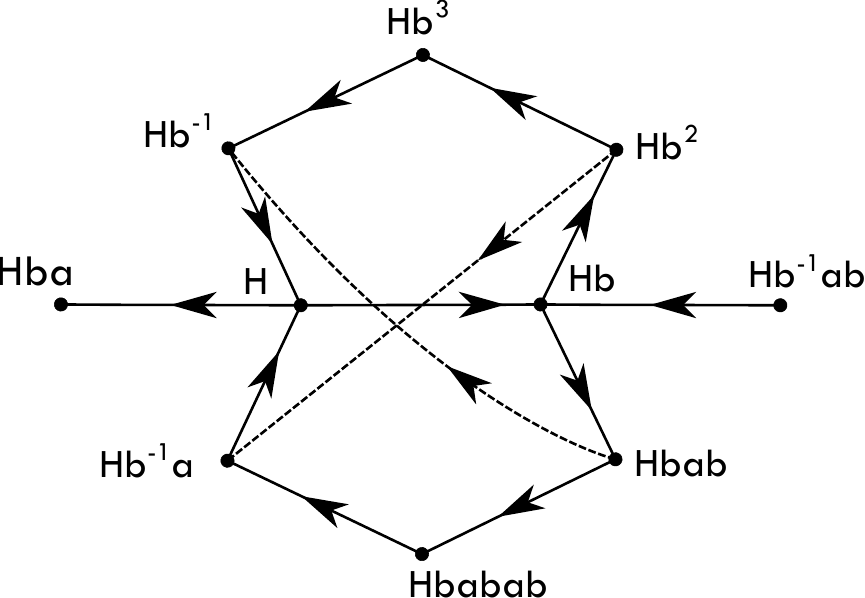}
    \caption{The two $5$-cycles through the edge $\{H, Hb\}$.}%
    \label{fig:PSL}%
\end{figure}

We finally prove that $\G$ is in fact arc-transitive. To this end take the matrix
$$
	C = \begin{pmatrix} 1 & -1-2\xi-2\xi^2 \\ -1-2\xi-2\xi^2 & -1 \end{pmatrix}.
$$
Calculating we find that $(1+2\xi+2\xi^2)^2 = -3+4\xi^2+4\xi^3$, and so $C$ is invertible if and only if $-2+4\xi^2+4\xi^3 \neq 0$, which is easily seen to be the case. We can thus view $C$ as an element of $\mathrm{PGL}(2,p)$. Writing $c$ instead of $C$ it is then straightforward to verify that $c$, $ac$ and $bc$ are all involutions. This shows that $c$ commutes with $a$ and inverts $b$, and so we can let $c$ act by conjugation on the vertex set of $\G$. Since $c$ preserves $H$ and $\{b,b^{-1}\}$ this is clearly an automorphism of $\G$ and as it fixes the vertex $H$ and interchanges the neighbors $Hb$ and $Hb^{-1}$, the graph $\G$ is indeed arc-transitive, as claimed.
\end{proof}

In view of the discussion from the paragraph following the proof of Proposition~\ref{cay} and the fact that by the well-known Dirichlet's prime number theorem (see for instance~\cite[p. 167]{Lang}) there are infinitely many primes $p$ with $p \equiv 1 \pmod{10}$, we thus finally obtain the following.

\begin{theorem}
\label{the:oneortwo}
There exist infinitely many tetravalent graphs of girth $5$ admitting a half-arc-transitive group action such that each edge of the graph lies on a unique $5$-cycle, and there also exist infinitely many tetravalent graphs of girth $5$ admitting a half-arc-transitive group action such that each edge of the graph lies on precisely two $5$-cycles.
\end{theorem}

The above result shows that there is an abundance of examples of tetravalent graphs of girth $5$ admitting a half-arc-transitive group action. It thus makes sense to inspect the census of all tetravalent graphs admitting a half-arc-transitive group of automorphisms up to order $1000$~\cite{censusHAT} to see how many such graphs exist and how much variety one can get while observing some of their other interesting properties.

It turns out that, besides the Doyle-Holt graph, there are $39$ tetravalent half-arc-transitive graphs of girth $5$ up to order $1000$ (of which $26$ are tightly-attached) and all of them have just one $5$-cycle through each edge. Moreover, up to order $1000$ there are $68$ tetravalent arc-transitive graphs of girth $5$ admitting a half-arc-transitive subgroup of automorphisms. Of these $15$ have two $5$-cycles through each edge, while $52$ have just one $5$-cycle through each edge (the Rose window graph $\mathrm{R}_{12}(5,2)$ has five $5$-cycles through each edge).

In Table~\ref{tab:HAT} we list some of the tetravalent half-arc-transitive graphs of girth $5$ up to order $1000$ (we use the names from~\cite{censusHAT}) and for each of them indicate whether or not it is a Cayley graph (Cay), has a solvable automorphism group (solv) and also give its radius (rad), attachment number (att), lengths of consistent cycles (con\_cyc), as well as an indication of whether or not a corresponding oriented graph is alter-complete (alt-cmp). Given that for the tightly-attached examples most of these are clear, we only include two tightly-attached examples in this table.
\begin{table}[ht]
%\begin{tabular}{ |p{2 cm}||p{2cm}|p{2cm}|p{2cm}| p{2cm}| p{2cm}| p{2cm}| }
\begin{center}
\begin{tabular}{|c||c|c||c|c|c|c|}
\hline
% name              &  rad & att & con\_cyc & Cay & solv & alt-cmp \\
name              & Cay & solv & rad & att & con\_cyc & alt-cmp \\
%  Cay   &   meta    &  $R^+$  &   solv   &   rad      &  AT \\
\hline
\hline
%HAT[55,1]        &     yes   &    yes      &   yes       &   yes   &   $11$    & TA \\
HAT[55,1]	& yes & yes & 11 & 11 & 5, 10 & no \\
 \hline
%HAT[480,44]     &    no    &     no      &   yes       &    no   &    $5$     & LA \\
HAT[480,44]	& no & no & 5 & 1 & 5, 6 & yes \\
 \hline
HAT[605,3]	& yes & yes & 11 & 1 & 5, 10 & no \\
 \hline
HAT[605,7]	& yes & yes & 121 & 121 & 5, 10 & no \\
\hline
%HAT[880,1]      &    yes    &    no       &  no        &   yes    &   $22$    &  O \\
HAT[880,1]	& yes & yes & 22 & 11 & 5, 10 & no \\
 \hline
%HAT[960,154]   &   yes    &     no      &   yes       &    no    &    $6$    & LA  \\
HAT[960,154]	& yes & no & 6 & 1 & 5, 8 & yes \\
\hline
HAT[960,162]	& no & no & 10 & 2 & 5, 6 & yes \\
 \hline
\end{tabular}
\caption{Examples of tetravalent half-arc-transitive graphs of girth $5$ and their properties.\label{tab:HAT}}
\end{center}
\end{table}

In Table~\ref{tab:nHAT1} we list various examples of tetravalent arc-transitive graphs of girth $5$ with one $5$-cycle through each edge and which admit a half-arc-transitive subgroup of automorphisms. Again we give the information on whether the graph is Cayley or not and whether its automorphism group is solvable or not. This time there can be different half-arc-transitive subgroups of automorphisms (by Corollary~\ref{stab4} they all have vertex stabilizers of order $2$) giving different orientations of the edges, and so for each half-arc-transitive group of automorphisms we give the various information for a corresponding orientation of the edges, just as in Table~\ref{tab:HAT}. 
%{\color{red} ??? One can show that ???} in the case that we indeed have two half-arc-transitive subgroups, say $H$ and $G$, each of $H$ and $G$ has one orbit of consistent $5$-cycles while the other $H$-orbit of $H$-consistent cycles consists of the $G$-alternating cycles and the other $G$-orbit of $G$-consistent cycles consists of the $H$-alternating cycles. 
%
\begin{table}[ht]
%\begin{tabular}{ |p{2 cm}||p{2cm}|p{2cm}|p{2cm}| p{2cm}| p{2cm}| p{2cm}| }
\begin{center}
\begin{tabular}{|c||c|c||c|c|c|c|}
\hline
name              & Cay & solv & rad & att & con\_cyc & alt-cmp \\
%  Cay   &   meta    &  $R^+$  &   solv   &   rad      &  AT \\
\hline
\hline
GHAT[60,10]	& yes & no & 3 & 2 & 5, 10 & yes \\
  \hdashline
			&  &  & 5 & 2 & 5, 6 & yes \\
 \hline
GHAT[80,12] & no & yes & 4 & 2 &  5, 10 & no \\
 \hline
GHAT[125,2] & yes & yes & 5 & 1 & 5, 10 & no \\
\hline
GHAT[150,8] & no & no & 5 & 1 & 5, 15  & no\\
\hline
GHAT[160,22]& yes & yes & 4 & 2 & 5, 10 & no \\
 \hdashline
			&  &  & 5 & 2 & 5, 8 & yes \\
\hline
GHAT[300,20]& yes & no & 3 & 1 & 5, 20 & no \\
 \hdashline
			&  &  & 10 & 5 & 5, 6 & yes \\
\hline
GHAT[540,48] & no & no & 5 & 1 & 5, 15  & yes\\
\hline
GHAT[800,83]& yes & yes & 10 & 2 & 5, 40 & no \\
 \hdashline
			&  &  & 20 & 10 & 5, 20 & no \\
\hline
\end{tabular}
\caption{Examples of tetravalent arc-transitive graphs of girth $5$ having one $5$-cycle through each edge and admitting a half-arc-transitive group action.\label{tab:nHAT1}}
\end{center}
\end{table}

We finally focus on examples having two $5$-cycles through each edge. In this case the fact that there are precisely two $5$-cycles through each edge and that in a corresponding orientation given by a half-arc-transitive subgroup of automorphisms all $5$-cycles are directed implies that prescribing the orientation of a particular edge uniquely determines the orientation of all the edges incident to it, and then inductively of each edge of the graph. Therefore, there can be at most one pair of orientations given by a half-arc-transitive action, and so if there exist at least two half-arc-transitive subgroups of automorphisms then there are precisely two of them and one is a subgroup of the other. In this case Corollary~\ref{stab4} and Proposition~\ref{consistent} imply that the smaller group has vertex stabilizers of order $2$ and only has consistent $5$-cycles while the larger one has vertex stabilizers of order $4$ and has one orbit of consistent $5$-cycles and one orbit of consistent $d$-cycles for some $d > 5$. 

In Table~\ref{tab:nHAT2} we list some of the $15$ tetravalent graphs of order at most $1000$ and girth $5$ with two $5$-cycles through each edge which admit a half-arc-transitive group action. We again include similar information as in the above two tables where this time we also include the information on the order of the vertex stabilizers (stab) of a corresponding half-arc-transitive subgroup. Here we make the agreement that if there are two half-arc-transitive subgroups of automorphisms we write $2,4$ in the `stab' column and $5, d$ in the `con\_cyc' column (see the above comment). Since by Proposition~\ref{cay} none of these graphs is a Cayley graph we omit this information from the table.
\begin{table}[ht]
%\begin{tabular}{ |p{2 cm}||p{2cm}|p{2cm}|p{2cm}| p{2cm}| p{2cm}| p{2cm}| }
\begin{center}
\begin{tabular}{|c||c|c||c|c|c|c|}
\hline
name              & stab & solv & rad & att & con\_cyc & alt-cmp \\
%  Cay   &   meta    &  $R^+$  &   solv   &   rad      &  AT \\
\hline
\hline
GHAT[30,4]	& $2, 4$ & no & $3$ & $2$ & $5, 6$ & yes \\
 \hline
GHAT[80,10] & $2, 4$ & yes & $4$ & $2$ & $5, 20$ & no \\
 \hline
GHAT[150,9] & $2, 4$ & no & $3$ & $1$ & $5, 30$  & no\\
\hline
GHAT[165,4] & $4$ & no & $3$ & $1$ & $5, 11$ & yes \\
\hline
GHAT[480,114] & $2$ & no & $6$ & $1$ & $5$ & yes \\
\hline
GHAT[640,1] 	& $2, 4$ & yes & $4$ & $1$ & $5, 10$ & no \\
\hline 
GHAT[825,8] & $4$ & no & $3$ & $1$ & $5, 55$ & no \\
\hline
\end{tabular}
\caption{Examples of tetravalent arc-transitive graphs of girth $5$ having two $5$-cycles through each edge and admitting a half-arc-transitive group action.\label{tab:nHAT2}}
\end{center}
\end{table}

The above data suggests that there is a great variety in examples of tetravalent graphs of girth $5$ admitting a half-arc-transitive group action. It also seems to suggest that the examples with two $5$-cycles through each edge are not as common as those with a unique $5$-cycle through each edge. But the one thing that stands out the most is that (at least up to order $1000$) none of the examples with two $5$-cycles through each edge is half-arc-transitive. It thus makes sense to ask if this is the case in general. Now, suppose $\G$ is a tetravalent graph of girth $5$ with two $5$-cycles through each edge and suppose $\G$ admits a half-arc-transitive subgroup $G \leq \Aut(\G)$. Let $v$ be a vertex of $\G$, let $u$ be one of its neighbors and choose the $G$-induced orientation $\vG$ such that $v \to u$. By Proposition~\ref{consistent} the $5$-cycles are $G$-consistent, and so there must exist elements $g,g' \in G$, each of which is of order $5$ and maps $v$ to $u$ (these are the $G$-shunts of the two $G$-consistent $5$-cycles through the edge $vu$). It is easy to see that $z = g'g^{-1}$ is an involution fixing $v$ but interchanging the two out-neighbors of $v$ as well as the two in-neighbors of $v$ in $\vG$. Therefore, $H = \langle g, z \rangle$ is a half-arc-transitive subgroup of $G$. If $H \neq G$, then $G$ has vertex stabilizers of order $4$ and $H$ is an index $2$ half-arc-transitive subgroup of $G$. Moreover, if $H = G$ and $G$ has vertex stabilizers of order $4$ then $G_v = \langle z, z' \rangle$ for some involution $z'$, $z' \neq z$, commuting with $z$. In any case, the group $H$ acts half-arc-transitively on $\G$, is a quotient group of the finitely presented group 
$$
	\tilde{G} = \langle \tilde{z}, \tilde{g} \mid \tilde{z}^2, \tilde{g}^5, (\tilde{z}\tilde{g})^5 \rangle
$$
and a vertex stabilizer in $H$ is either generated by the element corresponding to $\tilde{z}$ or by the element corresponding to $\tilde{z}$ and some involution commuting with it. 

One can thus use the {\tt LowIndexNormalSubgroups} method (see~\cite{ConDob05}) in {\sc Magma}~\cite{magma} to construct all candidates for such graphs up to some reasonably large order. It turns out that up to order $5\,000$ there are $43$ tetravalent graphs of girth $5$ with two $5$-cycles through each edge admitting a half-arc-transitive group action (recall that there are $15$ of them up to order $1000$). Of these there are $6$ examples admitting only one half-arc-transitive group of automorphisms where this group has vertex stabilizers of order $2$ (like the graph GHAT[480,114]), there are $8$ examples admitting only one half-arc-transitive group of automorphisms where this group has vertex stabilizers of order $4$ (like the graph GHAT[165,4]), while the remaining $29$ examples admit two half-arc-transitive groups, a larger one with vertex stabilizers of order $4$ and the smaller one being its index $2$ subgroup. However, all of these $43$ examples are in fact arc-transitive. 

As much as the above data supports a conjecture that there are no tetravalent half-arc-transitive graphs of girth $5$ with two $5$-cycles through each edge such a conjecture would be false. Namely, it turns out that there exists a tetravalent half-arc-transitive graph of order $6480$ and girth $5$ with two $5$-cycles through each edge and that it is the unique smallest example with this property. Its automorphism group is a solvable group of the form 
$$
	(\ZZ_3^4 \rtimes ((\ZZ_2 \times Q_8) \rtimes \ZZ_2)) \rtimes \ZZ_5
$$
and has vertex stabilizers of order $2$. Given that the graph can relatively easily be obtained via the above mentioned {\tt LowIndexNormalSubgroups} method while giving the group as a finitely presented group would require a rather long list of relations, we do not spell out the automorphism group of the graph explicitly. We do mention that the radius and attachment number of the graph are $4$ and $1$, respectively, and that the graph is not alter-complete. 

The above data suggests many interesting questions to be considered in possible future studies. One of them is for sure the following one (note that the potential graphs giving an affirmative answer to the question necessarily have two $5$-cycles through each edge).

\begin{question}
Does there exist a tetravalent half-arc-transitive graph of girth $5$ whose automorphism group has vertex stabilizers of order $4$? If so, what is the smallest example? 
\end{question}

To mention just one more interesting question observe that inspecting Table~\ref{tab:nHAT1} it appears that if a tetravalent graph $\G$ of girth $5$ with a unique $5$-cycle through each edge admits more than one half-arc-transitive subgroup of automorphisms then it admits precisely two. It turns out that up to order $1000$ for each tetravalent graph of girth $5$ with a unique $5$-cycle through each edge the following holds: if this graph admits more than one half-arc-transitive subgroup of automorphisms then the vertex stabilizers in the full automorphism group are of order $4$. Of course, if this is indeed the case then any half-arc-transitive subgroup is of index $2$ and is thus normal, in which case one easily shows that there can be at most two half-arc-transitive subgroups. What is more, one can show that if $H$ and $G$ are two half-arc-transitive subgroups (each of which is of index $2$ in the full automorphism group) then the $H$-consistent cycles which are not $5$-cycles are the $G$-alternating cycles while the $G$-consistent cycles which are not $5$-cycles are the $H$-alternating cycles (which is in accordance with the data in Table~\ref{tab:nHAT1}). However, it is not obvious why the stabilizers in the full automorphism group should be of order $4$ if we have at least two half-arc-transitive subgroups. To support this we mention that, although the example GHAT[125,2] from Table~\ref{tab:nHAT1} does admit just one half-arc-transitive subgroup of automorphisms, its full automorphism group has vertex stabilizers of order $8$, despite the fact that there is just one $5$-cycle through each edge of the graph. We thus pose the following question.

\begin{question}
Does there exist a tetravalent graph of girth $5$ admitting more than two half-arc-transitive subgroups of its automorphism group?
\end{question}
\end{section}

%%%%%%%%%%%%%%%%%%%%%%%%%%%%%%%%%

\begin{section}{An application: weak metacirculants of girth $5$}

We conclude the paper by showing how the results from the previous sections can be used to easily obtain a classification of tetravalent half-arc-transitive weak metacirculants of girth $5$. It was shown in~\cite{C2C4} that each non-tightly-attached tetravalent half-arc-transitive weak metacirculant is of Class~IV. In view of Theorem~\ref{TA} we thus only need to consider the weak metacirculants of Class~IV. It was shown in~\cite[Theorem~3.1]{ClassIV} that the tetravalent half-arc-transitive weak metacirculants of Class~IV all belong to the $8$-parametric family of graphs $\mathcal{X}_{IV}(m,n;r,t;p,a;q,b)$, the definition of which we review for self completeness (but see~\cite{ClassIV} for details). 

The graph $\mathcal{X}_{IV}(m,n;r,t;p,a;q,b)$, where $m \geq 5$, $n \geq 3$, $1 \leq p < q < m/2$ and $r,t,a,b \in \ZZ_n$ are such that 
$$
\gcd(p,q,m) = 1,\ \gcd(a,b,t,n) = 1,\ \gcd(t,n) \neq 1,\ r^m = 1\ \text{and}\ t(r-1)=0,
$$
has vertex set $\{u_i^j \colon i \in \ZZ_m,\ j \in \ZZ_n\}$ and adjacencies of the form
$$
	 u_i^j\>\> \sim \>\>
\begin{cases}
  u_{i+p}^{j+a r^i}, \,  u_{i+q}^{j+b r^i}  & ; \quad \text{$0\leq  i<m-q$, $j\in\mathbb{Z}_n$ }\\
   u_{i+p}^{j+a r^i}, \, u_{i+q}^{j+b r^i +t}  & ; \quad \text{$m-q\leq  i<m-p$, $j\in\mathbb{Z}_n$  }\\
   u_{i+p}^{j+a r^i+t}, \,   u_{i+q}^{j+b r^i +t}  & ; \quad \text{$m-p\leq i<m$, $j\in\mathbb{Z}_n.$  }\\
\end{cases}		
$$
The edges connecting vertices with subscripts of the form $i$ and $i+p$ (addition performed modulo $m$) are called the {\em $p$-edges} and those connecting vertices with subscripts of the form $i$ and $i+q$ are {\em $q$-edges}. It follows from~\cite[Theorem~3.1]{ClassIV} that each tetravalent half-arc-transitive $\mathcal{X}_{IV}(m,n;r,t;p,a;q,b)$ graph admits a regular subgroup of automorphisms (generated by the automorphisms $\rho$ and $\sigma$ from~\cite{ClassIV}) whose two orbits on the set of the edges of the graph are the set of all $p$-edges and the set of all $q$-edges. Using our results we now show that such a graph cannot be of girth $5$.

\begin{proposition}
There exist no tetravalent half-arc-transitive weak metacirculants of Class~IV and girth $5$.
\end{proposition}

\begin{proof}
Let us assume that $\G = \mathcal{X}_{IV}(m,n;r,t;p,a;q,b)$ is a tetravalent half-arc-transitive weak metacirculant of Class~IV and girth $5$. Since $\mathrm{R}_{12}(5,2)$ is arc-transitive while $\mathcal{X}_o(3,9;4)$ is not of Class~IV (as can easily be verified, but see for instance~\cite{C2C4}), Theorem~\ref{neusmerjeni} implies that the $5$-cycles of $\G$ are all directed. As explained in the comments preceding this lemma the results of~\cite{ClassIV} show that $\G$ is a Cayley graph, and so Corollary~\ref{stab4} and Proposition~\ref{cay} imply that each edge of $\G$ lies on a unique $5$-cycle, that the vertex stabilizers in $\Aut(\G)$ are of order $2$,  and that one of the two $5$-cycles through the vertex $u_0^0$ consists of five $p$-edges while the other consists of five $q$-edges. Moreover, for each vertex $v$ of $\G$ the unique nontrivial element of the stabilizer $\Aut(\G)_v$ interchanges the set of all $p$-edges with the set of all $q$-edges (as the regular group mentioned in the paragraph preceding this lemma is of index $2$ in $\Aut(\G)$). 

As $1 \leq p < q < m/2$, this implies that $5p = m$ and $5q = 2m$, forcing $q = 2p$. The condition $\gcd(p,q,m) = 1$ thus forces $p = 1$, and consequently $m=5$ and $q=2$. It is not difficult to see (but see~\cite{ClassIV}) that because of $p = 1$ we can in fact assume that $a = 0$, and so to obtain a $5$-cycle through $u_0^0$ consisting of five $p$-edges we require that $t = 0$. Then $\gcd(a,b,t,n) = 1$ implies $\gcd(b,n) = 1$, and one can again easily show that we can in fact assume that $b = 1$. Therefore, $\G = \mathcal{X}_{IV}(5,n;r,0;1,0;2,1)$ with $r^5 = 1$, $r \neq 1$ (otherwise $\G$ would have $4$-cycles). To obtain the $5$-cycle through $u_0^0$ consisting of five $q$-edges we thus require $1+r+r^2+r^3+r^4 = 0$. 

Let $\tau$ be the unique nontrivial automorphism fixing $u_0^0$ and recall that it swaps $p$-edges with $q$-edges. Consider the $\tau$-image of the walk $(u_0^0,u_2^1,u_1^1,u_0^1,u_2^2,u_1^2,u_0^2,u_2^3,\ldots , u_0^5)$ of length $15$ (on which every third edge is a $q$-edge and the remaining ones are $p$-edges). Fix the $\Aut(\G)$-induced orientation of the edges of $\G$ in which $u_0^0 \to u_1^0$. If $u_0^0 \to u_2^1$, then one readily verifies that $\tau$ maps $u_0^1$ to $u_2^{-r^2-r^4}$, $u_0^2$ to $u_4^{-r-r^2-2r^4} = u_4^{1+r^3-r^4}$, $u_0^3$ to $u_1^{1-r-r^4}$, $u_0^4$ to $u_3^{-r-r^3-r^4}$ and finally $u_0^5$ to $u_0^{-1-r-r^2-r^3-r^4} = u_0^0$. Since $u_0^0$ is left fixed by $\tau$ this implies $n = 5$. But this contradicts the fact that $r^5 = 1$ and $r \neq 1$. In an analogous way one can verify that if $u_2^1 \to u_0^0$ we again have that $\tau$ maps $u_0^5$ to $u_0^0$, which we already know cannot be the case. 
\end{proof}

Together with Theorem~\ref{TA} this yields the following result.

\begin{theorem}\label{the:meta}
Let $\G$ be a tetravalent half-arc-transitive weak metacirculant of girth $5$. Then $\G$ is tightly-attached and is either isomorphic to the Doyle-Holt graph $\mathcal{X}_o(3,9;4)$ or is isomorphic to $\mathcal{X}_o(5, r; q)$ for an odd integer $r$ and some $q \in \ZZ_r^*$ with $q \neq \pm 1$ and $1+q+q^2+q^3+q^4 = 0$. 
\end{theorem}

 \end{section}


\markboth{}{}

\begin{thebibliography}{}

\bibitem{ClassIV} I.~Anton\v ci\v c, P.~\v Sparl, 
		Classification of quartic half-arc-transitive weak metacirculants of girth at most $4$, 
		{\em Discrete Math.} {\bf 339} (2016), 931--945.

\bibitem{CC} M.~Boben, \v S.~Miklavi\v c, P.~Poto\v cnik, 
		Consistent cycles in a half-arc-transitive graphs,
		{\em The Electronic J. of Combin.} {\bf 16} (2009), 1--10.

\bibitem{magma} W.~Bosma, J.~Cannon, C.~Playoust, 
		The Magma algebra system. I. The user language, 
		{\em J. Symbolic Comput.} {\bf 24} (1997), 235–-265.

\bibitem{ConDob05} M.~Conder, P.~Dobcs\'anyi, 
		Applications and adaptations of the low index subgroups procedure,
		{\em Math. Comp.} {\bf 74} (2005), 485--497.

\bibitem{CuiZho21} L.~Cui, J.-X.~Zhou,
		A classification of tetravalent half-arc-transitive metacirculants of 2-power orders,
		{\em Appl. Math. Comput.} {\bf 392} (2021), Paper No. 125755, 14 pp.

\bibitem{DF04book} D.~S.~Dummit, R.~M.~Foote,
		``Abstract algebra'',
		J.~Wiley and Sons, Inc., Hoboken, 2004.
		
\bibitem{GR01book} C.~Godsil and G.~Royle, 
		``Algebraic graph theory'',
		Springer-Verlag, New York, 2001.

\bibitem{dickson} B.~Huppert.
		Endliche gruppen I. Springer Verlag.
		Berlin, Heidelberg, New York, 1967.

\bibitem{KovKutMar10} I.~Kov\'acs, K.~Kutnar, D.~Maru\v si\v c,
		Classification of edge-transitive rose window graphs,
		{\em J. Graph Theory} {\bf 65} (2010), 216--231.
		
\bibitem{RoseWindow} I.~Kov\'acs, K.~Kutnar, J.~Ruff,
		Rose window graphs underlying rotary maps, 
		{\em Discrete Math.} {\bf 310} (2010), 1802--1811.

\bibitem{Lang} S.~Lang,
		Algebraic number theory, 
		Springer Verlag, New York, 1986.
		
\bibitem{MalMarSeiSpaZgr08} A.~Malni\v c, D.~Maru\v si\v c, N.~Seifter, P.~\v Sparl, B.~Zgrabli\'c, 
		Reachability relations in digraphs, 
		{\em European J. Combin.} {\bf 29} (2008), 1566--1581.
		
\bibitem{MalPotSeiSpa15} A.~Malni\v c, P.~Poto\v cnik, N.~Seifter, P.~\v Sparl, 
		Reachability relations, transitive digraphs and groups, 
		{\em Ars Math. Contemp.} {\bf 8} (2015), 83--94.
		
\bibitem{MTAO} D.~Maru\v si\v c, 
		Half-arc-transitive group actions on finite graphs of valency $4$, 
		{\em J. Combin. Theory, Ser. B} {\bf 73} (1998), 41--76.

\bibitem{HTG4} D.~Maru\v si\v c, R.~Nedela, 
		Finite graphs of valency $4$ and girth $4$ admitting half-arc-transitive group actions, 
		{\em J. Austral. Math. Soc.} {\bf 73} (2002), 155--170.

\bibitem{MarPot02} D.~Maru\v si\v c, P.~Poto\v cnik,
		Bridging semisymmetric and half-arc-transitive actions on graphs,
		{\em European J. Combin.}  {\bf 23} (2002), 719--732.
		
\bibitem{MarPra99} D.~Maru\v si\v c, C.~E.~Praeger, 
		Tetravalent graphs admitting half-transitive group actions: alternating cycles, 
		{\em J.~Combin.~Theory Ser.~B} {\bf 75} (1999), 188--205.

\bibitem{classes} D.~Maru\v si\v c, P.~\v Sparl, 
		On quartic half-arc-transitive metacirculants, 
		{\em J. Algebraic. Combin.} {\bf 28} (2008), 365--395.
		
\bibitem{MarXu97} D.~Maru\v si\v c, M.-Y.~Xu, 
		A $\frac{1}{2}$-transitive graph of valency $4$ with a nonsolvable group of automorphisms, 
		{\em J. Graph Theory} {\bf 25} (1997), 133--138.

\bibitem{censusHAT} P.~Poto\v cnik, P.~Spiga, G.~Verret,
		A census of 4-valent half-arc-transitive graphs and arc-transitive digraphs of valence two, 
		{\em Ars Math. Contemp.} {\bf 8} (2015), 133--148.
		
\bibitem{PotSpa17} P.~Poto\v cnik, P.~\v Sparl,
		On the radius and the attachment number of tetravalent half-arc-transitive graphs,
		{\em Discrete Math.} {\bf 340} (2017), 2967--2971.
		
\bibitem{PotWil07} P.~Poto\v cnik, S.~Wilson,
		Tetravalent edge-transitive graphs of girth at most 4,
		{\em J. Combin. Theory, Ser. B} {\bf 97} (2007), 217--236.

\bibitem{PotWil} P.~Poto\v cnik, S.~Wilson, 
		Recipes for edge-transitive tetravalent graphs, 
		{\em The Art of Discrete and Applied Math.} {\bf 3} (2020), 1--33.
		
\bibitem{PozPra21} N.~Poznanovi\'c, C.~E.~Praeger, 
		Four-valent oriented graphs of biquasiprimitive type,
		{\em Algebr. Comb.} {\bf 4} (2021), no. 3, 409--434.
		
\bibitem{RamSpa19} A.~Ramos Rivera, P.~\v Sparl,
		New structural results on tetravalent half-arc-transitive graphs,
		{\em J. Combin. Theory, Ser. B}, {\bf 135} (2019), 256--278.

\bibitem{TA4E} P.~\v Sparl, 
		A classification of tightly attached half-arc-transitive graphs of valency $4$, 
		{\em J. Combin. Theory, Ser. B} {\bf 98} (2008), 1076--1108.

\bibitem{OCM} P.~\v Sparl, 
		On the classification of quartic half-arc-transitive metacirculants, 
		{\em Discrete Math.} {\bf 309} (2009), 2271--2283.

\bibitem{C2C4} P.~\v Sparl, 
		Almost all quartic half-arc-transitive weak metacirculants of Class II are of Class IV, 
		{\em Discrete Math.} {\bf 310} (2010), 1737--1742.

\bibitem{Wil04} S.~Wilson,
		Semi-transitive graphs,
		{\em J. Graph Theory} {\bf 45} (2004), 1--27.
   
\bibitem{Wil08} S.~Wilson, 
		Rose window graphs,
		{\em Ars Math. Contemp.} {\bf 1} (2008), 7--19.
    
\end{thebibliography}
\end{document}